\documentclass[11pt]{amsart}
\usepackage[english]{babel}
\usepackage[T1]{fontenc}    
\usepackage[utf8]{inputenc}  
\usepackage{fancyhdr}       
\usepackage{amsfonts}       
\usepackage{amssymb}       	
\usepackage{amsmath}        
\usepackage{amsthm}         
\usepackage{mathtools}    
\usepackage[all,cmtip]{xy}
\usepackage{mathrsfs}
\usepackage{graphicx}
\usepackage{xcolor}
\usepackage[shortlabels]{enumitem} 

\setenumerate[0]{label=(\alph*)}

\usepackage{bbm} 

\usepackage[left=2cm,head=13pt,
top=1.5cm,right=2cm, bottom=1.5cm]{geometry}
\usepackage{hyperref}       
\usepackage{todonotes}

\usepackage{cleveref}

\usepackage{tikz-cd} 

\usepackage{csquotes}       
\usepackage[style=trad-alpha]{biblatex} 
\newtheorem{theorem}[subsubsection]{Theorem} 
\newtheorem{conjecture}[subsubsection]{Conjecture}

\theoremstyle{definition}
\newtheorem{proposition}[subsubsection]{Proposition}

\newtheorem{condition}[subsubsection]{Condition}

\newtheorem{notation}[subsubsection]{Notation}

\newtheorem{question}[subsubsection]{Question}

\newtheorem{lemma}[subsubsection]{Lemma}
\newtheorem{corollary}[subsubsection]{Corollary}

\numberwithin{equation}{subsection}

\theoremstyle{remark}
\newtheorem{remark}[subsubsection]{Remark} 

\numberwithin{equation}{section}

\DeclareMathOperator{\id}{id}

\DeclareMathOperator{\im}{im}

\DeclareMathOperator{\Span}{span}

\DeclareMathOperator{\rk}{rk}

\DeclareMathOperator{\stab}{Stab}

\DeclareMathOperator{\cent}{Cent}

\DeclareMathOperator{\Sp}{Sp}
\DeclareMathOperator{\Aut}{Aut}

\DeclareMathOperator{\codim}{codim}

\DeclareMathOperator{\dr}{dR}
\DeclareMathOperator{\crys}{crys}
\DeclareMathOperator{\ord}{ord} 

\DeclareMathOperator{\diag}{diag} 
\DeclareMathOperator{\fil}{Fil}
\DeclareMathOperator{\Conj}{conj}
\DeclareMathOperator{\Hdg}{Hdg}
\DeclareMathOperator{\sbt}{Sbt} 
\DeclareMathOperator{\fr}{Fr} 
\DeclareMathOperator{\bruh}{Bruh}

\newcommand{\fkg}{\mathfrak{g}}

\newcommand{\fkm}{\mathfrak{m}}

\newcommand{\gr}{\text{Gr}}

\newcommand{\gzip}{\text{$G$-$\mathtt{Zip}$}} 

\newcommand{\gzipf}{\text{$G$-$\mathtt{ZipFlag}$}} 
\newcommand{\gozip}{\text{$G_0$-$\mathtt{Zip}$}} 

\newcommand{\zipf}{\text{-$\mathtt{ZipFlag}$}}
\newcommand{\std}{{\rm Std}}
\newcommand{\spin}{{\rm spin}}

\newcommand{\pizipb}{\pi_{\mathtt{Zip},\mathcal{B}}}
\newcommand{\piflagz}{\pi_{\mathtt{Flag},\mathtt{Zip}}}
\newcommand{\piflags}{\pi_{\mathtt{Flag},\text{Sbt}}}

\newcommand{\calo}{\mathcal{O}}

\newcommand{\call}{\mathcal{L}}
\newcommand{\calx}{\mathcal{X}}

\newcommand{\mz}{\ZZ}
\newcommand{\mq}{\QQ}
\newcommand{\mr}{\RR}

\newcommand{\mf}{\FF}
\newcommand{\mg}{\GG}

\newcommand{\rH}{\mathrm{H}} 
\newcommand{\GL}{\mathrm{GL}} 
\newcommand{\SL}{\mathrm{SL}} 
\newcommand{\SO}{\mathrm{SO}} 
\newcommand{\GSO}{\mathrm{GSO}} 
\newcommand{\GSp}{\mathrm{GSp}} 
\newcommand{\GSpin}{\mathrm{GSpin}} 
\newcommand{\GU}{\mathrm{GU}}
\newcommand{\U}{\mathrm{U}}

\newcommand{\EO}{\mathrm{EO}} 

\newcommand{\ha}{\mathrm{Ha}} 

\newcommand{\stefan}[1]{{\color{blue} (Stefan: #1)}}

\usepackage{hyperref}
\hypersetup{
    colorlinks=true,
    linkcolor=blue,
    filecolor=blue,      
    urlcolor=blue,
}
\usepackage{cleveref}
 
\providecommand{\keywords}[1]
{
  \small	
  \textbf{\textit{Keywords---}} {#1}
}

\newcommand{\GZip}{\mathop{\text{$G$-{\tt Zip}}}\nolimits}

\newcommand{\GspZip}{\mathop{\text{$\GSp(2g)$-{\tt Zip}}}\nolimits}
\newcommand{\GLnZip}{\mathop{\text{$\GL(n)$-{\tt Zip}}}\nolimits}

\newcommand{\GLVZip}{\mathop{\text{$\GL(V)$-{\tt Zip}}}\nolimits}

\newcommand{\GF}{\mathop{\text{$G$-{\tt ZipFlag}}}\nolimits}

\newcommand{\Asf}{\mathsf{A}}
\newcommand{\Bsf}{\mathsf{B}}
\newcommand{\Csf}{\mathsf{C}}
\newcommand{\Dsf}{\mathsf{D}}

\renewcommand{\AA}{\mathbf{A}}

\newcommand{\CC}{\mathbf{C}}

\newcommand{\FF}{\mathbf{F}}
\newcommand{\GG}{\mathbf{G}}

\newcommand{\NN}{\mathbf{N}}

\newcommand{\PP}{\mathbf{P}}
\newcommand{\QQ}{\mathbf{Q}}
\newcommand{\RR}{\mathbf{R}}

\newcommand{\XX}{\mathbf{X}}

\newcommand{\ZZ}{\mathbf{Z}}

\newcommand{\Sscr}{{\mathscr S}}

\newcommand{\Vscr}{{\mathscr V}}

\newcommand{\Lcal}{{\mathcal L}}

\newcommand{\Ocal}{{\mathcal O}}

\newcommand{\eg}{e.g.,\ }

\newcommand{\ie}{i.e.,\ }
\newcommand{\fp}{\FF_p}
\newcommand{\leftexp}[2]{{\vphantom{#2}}^{#1}{#2}}

\newcommand{\iw}{\leftexp{I}{W}}
\newcommand{\clp}{\mathsf{clp}}
\newcommand{\dR}{{\rm dR}}
\newcommand{\ad}{\textnormal{ad}}
\newcommand{\qp}{\QQ_p}
\newcommand{\gofqp}{\GG(\qp)}
\newcommand{\af}{\AA_f}
\newcommand{\gofafp}{\GG(\AA_f^p)}
\newcommand{\gx}{(\GG, \XX)}
\newcommand{\type}{\mathsf{type}}

\title[An Ogus Principle for Zip period maps]{An Ogus Principle for Zip period maps: The Hasse invariant's vanishing order via `Frobenius and the Hodge filtration'}
\author[W. Goldring and S. Reppen]{Wushi Goldring, Stefan Reppen}
\address[W. Goldring]{Department of Mathematics, Stockholm University}
\address[S. Reppen]{\parbox{\linewidth}{Department of Mathematics, Stockholm University \\
Graduate School of Mathematical Sciences, the University of Tokyo}}
\email[W. Goldring, S. Reppen]{wushijig@gmail.com, stefan.reppen@gmail.com}

\makeatletter
\let\c@equation=\c@subsubsection

\let\c@figure=\c@subsubsection

\begin{document}

\maketitle

\begin{abstract}
This paper generalizes a result of Ogus that, under certain technical conditions, the vanishing order of the Hasse invariant of a  family $Y/X$ of $n$-dimensional Calabi-Yau varieties in characteristic $p$ at a point $x$ of $X$ equals the "conjugate line position" of $H^n_{\dR}(Y/X)$ at $x$, i.e. the largest $i$ such that the line of the conjugate filtration is contained in $\fil^i$ of the Hodge filtration. 
For every triple $(G,\mu,r)$ consisting of a connected, reductive $\fp$-group $G$, a cocharacter $\mu \in X_*(G)$ and an $\fp$-representation $r$ of $G$, we state a generalized Ogus Principle. 
If $\zeta:X \to \GZip^{\mu}$ is a smooth morphism (=`Zip period map'), then the group theoretic Ogus Principle implies an Ogus Principle on $X$. We deduce an Ogus Principle for several Hodge and abelian-type Shimura varieties and the moduli space of K3 surfaces.
\end{abstract}

\tableofcontents
\section{Introduction}
This paper illustrates a novel example in the program by one of us (W. G.) to reinterpret and generalize algebro-geometric structure within the frameworks of \textit{G-Zip Geometricity} \cite{Goldring-Koskivirta-global-sections-compositio} joint with J.-S. Koskivirta and \textit{Geometry-by-groups} \cite{goldring.griffiths}. The algebro-geometric starting point of this example is a principle of Ogus \cite{ogus-Calabi-Yau} \cite{Ogus-height-strata-K3} describing the vanishing order of the Hasse invariant on the base of a  family of Calabi-Yau varieties in characteristic $p>0$ as the largest step of the (descending) Hodge filtration containing the line (=smallest step) of the (ascending) conjugate filtration. This `\textit{conjugate line position}' is by its very definition given by the $G$-Zip package. The same is true about the divisor of the Hasse invariant, which reduces to Chevalley's formula in Schubert calculus \cite[Theorem 2.2.1]{goldring.koskivirta.invent}. However, it is interesting that, via Ogus' Principle, the vanishing order of the Hasse invariant at singular points of its divisor is also given group-theoretically in terms of $G$-Zips. This latter kind of geometric information may at first seem more `genuinely geometric' and less amenable   to group-theoretic techniques. Indeed it was not previously considered in applications of $G$-Zips to algebraic geometry.

This paper centers on the group theoretic Ogus' principle, 
generalizing Ogus' Principle to the setting of triples $(G,\mu,r)$, where $G$ is a connected, reductive $\fp$-group, $\mu \in X_*(G)$ is a cocharacter over an algebraic closure and $r$ is an $\fp$-representation of $G$. Our group theoretic Ogus' principle simultaneously generalizes Ogus' original principle in at least three ways:
\begin{enumerate}
\item
\label{item-intro-abstract-hodge}
The generalization from geometric families $Y/X$ in characteristic $p$ to "abstract" stacks $\GZip^{\mu}$ of $G$-Zips is analogous to generalizing from geometric families $Y/X$ over $\CC$ to "abstract" variations of Hodge structure (VHS). It shows that Ogus' Principle makes sense for all $(G,\mu)$ even when there is no known geometric family $Y/X$ associated to some $(G,\mu)$.
\item 
\label{item-intro-weights-motivic}
The abstract setting~\ref{item-intro-abstract-hodge} highlights that Ogus' Principle also applies to geometric families $Y/X$ whose fibers are not Calabi-Yau (CY) varieties. 
What matters is that the cohomology (or equivalently the representation $r$) be of CY-type~\eqref{sec-CY-rep}. 
This illustrates that the CY-setting in Ogus' Principle concerns \textit{weights} (considered either Hodge-theoretically, motivically or representation-theoretically) rather than \textit{spaces} 
(\eg varieties).  
Building on the example of Hilbert modular varieties treated by one of us (S. R.) \cite{reppen1}, our fundamental example of considering more general geometric families $Y/X$ is when $X$ is the special fiber of an integral canonical model of a Hodge-type Shimura variety and $Y/X$ is a universal abelian scheme\footnote{Here we follow Ogus \cite{ogus-Calabi-Yau} by defining an $n$-dimensional variety $X$ to be Calabi-Yau if $\dim H^0(X, \Ocal_X)=\dim H^n(X,\Ocal_X)=1$ and $H^i(X, \Ocal_X)=0$ for all $i \neq 0,n$. 
The vanishing of $H^1(X,\Ocal_X)$ excludes abelian varieties.}.
\item
\label{item-intro-tannakian}
As already alluded to in~\ref{item-intro-weights-motivic}, stating Ogus' Principle for different groups $G$ and different representations $r$ connects it to representation theory and offers a Tannakian framework\footnote{The categories of $G$-Zips are not strictly speaking Tannakian since they are not abelian, for the same reason as for vector bundles, $G$-bundles etc.} within which to study it.
\end{enumerate}
The title pays homage to Mazur's pioneering paper "Frobenius and the Hodge filtration" \cite{Mazur-Frobenius-Hodge} which sowed the seeds for many of the developments considered here.
\subsection{\texorpdfstring{$G$}{G}-Zip Geometricity}
\label{sec-intro-zip-geom}
Let $p$ be a prime and let $k$ be an algebraic closure of $\fp$. Let $G$ be a connected, reductive $\fp$-group and let $\mu \in X_*(G)$ be a cocharacter over $k$. 
Associated to the pair $(G,\mu)$, Pink-Wedhorn-Ziegler \cite{pink.wedhorn.ziegler.zip.data} \cite{pink.wedhorn.ziegler.additional} define a stack $\GZip^{\mu}$ classifying $G$-Zips of type $\mu$. Morphisms $X \to \GZip^{\mu}$ give a mod $p$ analogue of Hodge structures with $G$-structure: 
Let $Y/X$ be a proper, smooth family of $k$-schemes whose Hodge-de Rham spectral sequence degenerates at $E_1$ and whose Hodge and de Rham cohomology sheaves are locally free. Then, for every $i \geq 0$, the de Rham cohomology $H^i_{\dR}(Y/X)$ is classified by a morphism $X \to \GLnZip^{\mu}$, where $n$ is the rank of $H^i_{\dR}(Y/X)$ and $\mu$ is deduced from the Hodge filtration, just like the Hodge-Deligne cocharacter in classical Hodge theory. 
Motivated by this analogy, we refer to morphisms of $k$-stacks $\zeta:X \to \GZip^{\mu}$ as \textit{Zip period maps} (including the case $X=\GZip^{\mu}$, $\zeta=\id$).  

Recall \cite[Question~A]{Goldring-Koskivirta-global-sections-compositio} that the question driving $G$-Zip geometricity is:
\begin{question}
\label{q-G-Zip-geom}
    Given a Zip period map 
\addtocounter{equation}{-1}
\begin{subequations}
\begin{equation}
 \label{eq-intro-zip-period}   
\zeta: X \to \GZip^{\mu},
\end{equation}    
\end{subequations}
 what geometry of $X$ is determined by properties of the morphism $\zeta$ and the stack $\GZip^{\mu}?$
\end{question}
Some successful examples of the philosophy of $G$-Zip geometricity include the works of Goldring-Koskivirta on strata Hasse invariants and their applications to the Langlands correspondence \cite{goldring.koskivirta.invent} and on the cone of global sections \cite{Goldring-Koskivirta-global-sections-compositio}, \cite{Goldring-Koskivirta-GS-cone}, \cite{Goldring-Koskivirta-divisibility}, 
the work of Brunebarbe-Goldring-Koskivirta-Stroh on ampleness of automorphic bundles \cite{Brunebarbe-Goldring-Koskivirta-Stroh-ampleness}, the works of Wedhorn-Ziegler \cite{Wedhorn-Ziegler-tautological} and Cooper \cite{Cooper-tautological-ring-hilbert} on tautological rings   
and the work of Cooper-Goldring  on strata-effectivity in these rings \cite{Cooper-Goldring-strata-eff}. 
\subsubsection{A key test-case: Hodge and abelian-type Shimura varieties} 
\label{sec-intro-Hodge-type}
Let $\gx$ be an abelian-type Shimura datum. 
Assume that $\GG$ is unramified at $p$.
Let $K_p \subset \gofqp$ be a hyperspecial maximal compact subgroup. 
By the work of Kisin \cite{kisin-hodge-type-shimura} and Vasiu \cite{vasiu}, as $K^p$ ranges over open, compact subgroups of $\gofafp$,  the associated projective system of Shimura varieties  admits an integral canonical model $(\Sscr_{K_pK^p}\gx)_{K^p}$ in the sense of Milne \cite{Milne-integral-canonical-models}. 
Set $K:=K_pK^p$ and let $S_K$ be the special $k$-fiber of $\Sscr_{K_pK^p}\gx$. 

Assume that $\gx$ is of Hodge type. For $g \geq 1$, let $(\GSp(2g), \XX_g)$ be the Siegel datum, consisting of the $\QQ$-split symplectic similitude group $\GSp(2g)$ and the Siegel double half-space $\XX_g$. Given a symplectic embedding
\addtocounter{equation}{-1}
\begin{subequations}
\begin{equation}
\label{eq-symplectic-embedding}
\gx \hookrightarrow (\GSp(2g), \XX_g), 
\end{equation}
for all sufficiently small $K^p$ there exists a level $K' \subset \GSp(2g, \af)$ and an induced finite map from $S_K$ to the special $k$-fiber of the Siegel-type Shimura variety $S_{g,K'}$  \cite[(2.3.3)]{kisin-hodge-type-shimura}. If $Y/S_K$ is the resulting family of abelian schemes, the Zip period map associated to $H^1_{\dR}(Y/S_K)$ factors through a smooth (Zhang \cite{zhang}) surjective (Kisin-Madapusi Pera-Shin \cite{Kisin-Madapusi-Pera-Shin-Honda-Tate}) morphism 

\begin{equation}
\label{eq-zeta-shimura}
\zeta:S_K \to \GZip^{\mu},
\end{equation}    
\end{subequations}
 where $G$ is the reductive $\fp$-group deduced from the $\QQ$-group $\GG$ and $\mu \in X_*(G)$ is a representative of the conjugacy class of cocharacters deduced from the Hermitian symmetric space $\XX$. See \cite[\S\S4.1-4.2]{goldring.koskivirta.invent} for more details. When $\gx$ is not of Hodge-type,  there still exists a smooth, surjective morphism~\eqref{eq-zeta-shimura} by \cite{shen.zhang}, but it no longer arises from an $F$-Zip of the form $H^1_{\dR}(Y/S_K)$.

\subsection{An Ogus Principle for Zip period maps
}
\label{sec-general-ogus}
\subsubsection{$F$-Zips, d'apr\`es Moonen-Wedhorn \cite{moonen.wedhorn}} When $G=\GL(n)$, a $G$-Zip of type $\mu$ corresponds to an $F$-Zip of rank $n$ and type $\mu$ as previously defined in \cite{moonen.wedhorn}. Recall that such an $F$-Zip on a stack $X$ is a quadruple 
\addtocounter{equation}{-1}
\begin{subequations}
\begin{equation}
\label{eq-intro-F-Zip}
\underline{\Vscr}=(\Vscr, \fil_{\Hdg}^{\bullet}, \fil^{\Conj}_{\bullet}, \varphi^{\bullet}),    
\end{equation}
\end{subequations}
where $\Vscr$ is a rank $n$ vector bundle on $X$, $\fil_{\Hdg}^{\bullet}\Vscr$ (resp. $\fil^{\Conj}_{\bullet}\Vscr$) is a descending (resp. ascending) filtration on $\Vscr$ by locally direct factors and $\varphi^{\bullet}$ is an isomorphism \textit{zipping} the Frobenius pullback of the descending graded pieces with the ascending graded pieces: 
$\varphi^{\bullet}:(\gr_{\Hdg}^{\bullet})^{(p)} \stackrel{\sim}{\to} \gr_{\bullet}^{\Conj}$. Following Deligne's sign convention for Hodge structures, the type $\mu \in X_*(\GL(n))$ is specified by the rule that $\rk \gr_{\Hdg}^i$ is the multiplicity of $-i$ as $\mu$-weight of $r$.
\subsubsection{The Hasse invariant and conjugate line position of an $F$-Zip of CY-type}
\label{sec-F-Zip-CY-type}
Let $i_{0}$ be the largest integer such that $\fil_{\Hdg}^{i_{0}}=\Vscr$. 
An $F$-Zip~\eqref{eq-intro-F-Zip} is of \underline{CY-type} if $\gr_{\Hdg}^{i_{0}}:=\fil_{\Hdg}^{i_0}/\fil_{\Hdg}^{i_0+1}$ is a line bundle.
The Hasse invariant of an $F$-Zip~\eqref{eq-intro-F-Zip} of CY-type is the composition 
\addtocounter{equation}{-1}
\begin{subequations}
  \begin{equation}
  \label{eq-def-hasse}
\ha(\underline{\Vscr}): (\gr_{\Hdg}^{i_{0}})^{(p)} \stackrel{\varphi^i}{\to} \gr_{i_{0}}^{\Conj} \to \Vscr \to \gr_{\Hdg}^{i_{0}}.    
\end{equation}
Equivalently, the Hasse invariant is a global section $\ha(\underline{\Vscr}) \in H^0(X, (\gr_{\Hdg}^{i_{0}})^{-(p-1)})$, since $\gr_{\Hdg}^{i_{0}}$ is a line bundle.
If $\underline{\Vscr}$ is not of CY-type and $d:= \rk \gr^{i_{0}}$, its Hasse invariant is the determinant (=$d$th exterior power) of~\eqref{eq-def-hasse}.
Equivalently, the Hasse invariant of such a $\underline{\Vscr}$ is the Hasse invariant~\eqref{eq-def-hasse} of the CY $F$-Zip $\wedge^d \underline{\Vscr}$.  Analogous to Ogus' case, the conjugate line position of a CY-F-Zip $\underline{\Vscr}$ at a $k$-point $x$ is defined by\footnote{In the literature, what we call the `conjugate line position' is sometimes called the `$a$-number'. We prefer the former since it is more descriptive.} 
\begin{equation}
\label{eq-intro-clp}
\clp_x(\underline{\Vscr}):=|\{j \in \ZZ | j \textnormal{ is a }\mu\textnormal{-weight of }r \textnormal{ and } \gr_{i_{0},x}^{\Conj} \subset \fil_{\Hdg,x}^j\}|-1,    
\end{equation}
\end{subequations}
the number of nontrivial pieces of the Hodge filtration containing the conjugate line $\gr_{i_{0},x}^{\Conj}$.
\subsubsection{Ogus' result}
\label{sec-ogus-result}
Let $Y/X$ be a proper smooth family of $n$-dimensional, Calabi-Yau $k$-varieties. Under certain technical conditions, including the degeneration of the Hodge-de Rham spectral sequence of $Y/X$, Ogus shows that for all $x \in X(k)$, the vanishing order and conjugate line position of the CY-$F$-Zip $H^n_{\dR}(Y/X)$ are equal:
\begin{equation}
\label{eq-intro-ogus for de rham}
\ord_x \ha(H^n_{\dR}(Y/X))=\clp_x(H^n_{\dR}(Y/X)).
\end{equation}
\subsubsection{CY-Representations}
\label{sec-CY-rep}
An $\fp$-representation $r:G \to \GL(V)$ is of \underline{CY-type} if the highest $\mu$-weight of $V_k$ has multiplicity one. Let $\underline{\Vscr}(G,\mu,r)$ be the universal $F$-Zip  of rank $\dim V$ and type $r \circ \mu$ over $\GLVZip^{r\circ \mu}$. If $r$ is of CY-type, define a Hasse invariant $\ha(G,\mu,r):=r^*\ha(\underline{\Vscr}(G,\mu,r))$ on $\GZip^{\mu}$ and conjugate line position $\clp(G,\mu,r):=\clp(\underline{\Vscr}(G,\mu,r))$. 
Let $\lambda$ be the highest weight of $V$ relative a maximal torus containing the image of $\mu$ and base of simple roots. Then $\ha(G,\mu,r)$ is a global section of the associated line bundle  $\Lcal_{\gzip}(-\lambda)^{p-1}$ on $\GZip^{\mu}$~\eqref{line.bundles.on.stacks}.
\subsubsection{Ogus' Principle for $(G,\mu,r)$} Let $r$ be of CY-type. Assume that the Hasse invariant $\ha(G,\mu,r)$ is not identically zero. Say that Ogus' Principle holds for $(G,\mu,r)$ if for all $x \in \GZip^{\mu}(k)$,
\addtocounter{equation}{-1}
\begin{subequations}
\begin{equation}
\label{eq-ogus-principle}
\ord_x \ha(G,\mu,r)=\clp_x(G,\mu,r).
\end{equation}
\end{subequations}
\subsubsection{Ogus' Principle for $(X,\zeta,r)$}
Given a Zip period map $\zeta$~\eqref{eq-intro-zip-period}, the conjugate line position is invariant under pullback.  If $\zeta$ is smooth then so is the Hasse invariant's vanishing order. So 
Ogus' Principle for $(G,\mu,r)$ implies an Ogus Principle 
for $(X,\zeta, r)$. If $\zeta$ arises from $H^i_{\dR}(Y/X)$ for some geometric family $Y/X$ as in~\ref{sec-intro-zip-geom}, then the Hasse invariant of the $F$-Zip $H^i_{\dR}(Y/X)$ on $X$ simultaneously agrees with both the classical definition of the Hasse invariant via de Rham cohomology (as in \cite{ogus-Calabi-Yau}, \cite{katz}) and the pullback of the Hasse invariant of the universal $F$-Zip over $\GLnZip^{\mu}$.
\begin{question}
 \label{q-intro-ogus}
 Which triples $(G,\mu,r)$ and $(X,\zeta, r)$ satisfy Ogus' Principle~\eqref{eq-ogus-principle}?
\end{question}

We establish several instances of Ogus' Principle, 
including many that apply to Shimura varieties $S_K$. 
Let $L:=\cent(\mu)$ be the Levi subgroup of $G_k$ centralizing $\mu$. When $G=\GL(n)$ or $G \subset \GL(n)$ is the (similitude) group of a bilinear form, let $\std:G \hookrightarrow \GL(n)$ denote the inclusion.  
\begin{theorem}
\label{th-main}
Ogus' Principle holds for the following triples $(G,\mu,r)$:
\begin{enumerate}
\item
\label{item-GL}
$G=\GL(n)$, $\type(L)=\Asf_{n-2}$ and $r=\wedge^n(\std \oplus \std^{\vee})$,
\item
\label{item-GSp}
$G=\GSp(2n)$, $\type(L)=\Asf_{n-1}$ and $r=\wedge^n \std^{\vee}$,
\item
\label{item-th-main-odd-orthogonal-Bn}
$G=\SO(2m+1)$, $\type(L)=\Bsf_{m-1}$ and $r=\std$ 
\item
\label{item-th-main-even-orthogonal-Dn}
$G=\SO(2m)$, $\type(L)=\Dsf_{m-1}$ and $r=\std$
\item 
\label{item-th-main-GL(4)}
$G=\GL(4)$, $\type(L)=\Asf_1 \times \Asf_1$ and $r=\wedge^2(\std)$.
\end{enumerate}   
\end{theorem}
In the orthogonal cases~\ref{item-th-main-odd-orthogonal-Bn}-\ref{item-th-main-even-orthogonal-Dn}, $\std$ is of CY-type. Its $\mu$-weights are $-1,0,1$ with multiplicities $1,2m-1,1$ for type $\Bsf_m$ and $1,2m-2,1$ for type $\Dsf_m$  respectively.

The method used to prove~\ref{th-main}, laid out in \S\ref{sec-intro-approaches}, also shows that Ogus' Principle fails but only so slightly in the following setting dual to~\ref{th-main}\ref{item-th-main-odd-orthogonal-Bn}, where $(G,\mu)$ does not arise from a Shimura variety.
\begin{theorem}
\label{th-intro-Cn-Cn-1}
Assume that $G=\Sp(2n)$, $\type(L)=\Csf_{n-1}$ 
 and $r=\std$. Then Ogus' Principle fails for $(G,\mu,r)$.  The divisor of the Hasse invariant $\ha(G,\mu,r)$ is reduced. Ogus' Principle holds for all $x \in \GZip^{\mu}(k)$, except that $\clp_x(G,\mu<r)=2$ for $x$ the closed point (=minimal zip stratum) of $\GZip^{\mu}$.
 \end{theorem}
 In view of~\ref{th-intro-Cn-Cn-1}, we ask:
 \begin{question}
 Which triples $(G,\mu,r)$ and $(X,\zeta, r)$ satisfy the inequality
 \addtocounter{equation}{-1}
 \begin{subequations}
 \begin{equation}
 \label{eq-intro-ineq}
 \ord_x \ha(G,\mu,r) \leq \clp_x(G,\mu,r)?   
 \end{equation}    
 \end{subequations}
 
 \end{question}
\begin{remark}
\label{rmk-ogus-assumptions}
We do not understand how our smoothness assumption on $\zeta$ compares with Ogus' assumptions. An advantage of our approach is that smoothness is a general condition, whereas according to Ogus, some of his assumptions are a priori more specialized to the setting of Calabi-Yau varieties. Both smoothness and Ogus' assumptions hold for K3 surfaces.   
\end{remark}

\begin{remark}
Prior to Ogus, it was a classical result of Igusa that the Hasse invariant of elliptic curves has simple zeros \cite{igusa}.
Recently Bhatt-Singh gave an upper bound for the Hasse invariant's vanishing order on certain families of Calabi-Yau hypersurfaces  \cite[Theorem 4.6]{bhatt.singh}. 
\end{remark}

\subsection{Ogus' Principle for Hodge-type Shimura varieties}
\label{sec-intro-app-Shimura}
Consider the key test-case~\ref{sec-intro-Hodge-type} of a Hodge-type Shimura variety $S_K$ and a universal family of abelian schemes $Y/S_K$ associated to a symplectic embedding $\varphi$~\eqref{eq-symplectic-embedding}.  The $F$-Zip $H^g_{\dR}(Y/S_K)=\wedge^gH^1_{\dR}(Y/S_K)$ is of CY-type. It arises from the representation $r$ which is the $g$th exterior power of the dual of $G \to \GSp(2g) \to \GL(2g)$ deduced from $\varphi$. The classical Hasse invariant on $S_K$ is the Hasse invariant of the $F$-Zip $H_{\dr}^{g}(Y/S_K)$ and equals $\zeta^*\ha(G,\mu,r)$. On the other hand, a result of Wedhorn \cite{Wedhorn-ordinariness-Shimura-varieties}, generalized by Wortmann \cite{Wortmann-mu-ordinary} characterises when the Hasse invariant of $S_K$ is identically zero. Based on \ref{th-main}, we make the following:  
\begin{conjecture}
\label{conj-shimura} Assume that the conjugacy class of $\mu$ is defined over $\fp$.
For every $x\in S_K(k)$, 
$$\ord_x \ha(H^g_{\dR}(Y/S_K))=\clp_x H^g_{\dR}(Y/S_K).
$$
\end{conjecture}
The main result~\ref{th-main} coupled with the smoothness of $\zeta$ for $S_K$~\eqref{eq-zeta-shimura} implies:
\begin{corollary}
\label{cor-intro-Hodge-type}
Conjecture~\ref{conj-shimura} holds for $S_K$ if the $\QQ$-group $\GG$ is either of the following:
\begin{enumerate}
    \item\label{item: cor hodge type unitary}
    A unitary similitude group associated to an imaginary quadratic field $K/\QQ$ such that the base change $\GG_{\RR}$ to $\RR$ has signature $(n-1,1)$ or $(2,2)$ and $p$ splits in $K$ in the $(n-1,1)$-case.
    \item
    \label{item: cor hodge type siegel}
    The $\QQ$-split symplectic group (the Siegel case).
    
\end{enumerate}
\end{corollary}
When $G=\GSpin(2n+1)$ (resp. $G=\GSpin(2n)$) is a spin similitude group of type $\Bsf_n$ (resp. $\Dsf_n)$ and $\type(L)=\Bsf_{n-1}$ (resp. $\type(L)=\Dsf_{n-1}$), we also explicitly  determine the vanishing order of the Hasse invariant $\ha(G,\mu,r)$ at every point of $\GZip^{\mu}$ for $r$ the spin representation. By pullback along $\zeta$~\eqref{eq-zeta-shimura} this gives the corresponding result for $\GSpin$ Hodge-type Shimura varieties $S_K$.

When $S_K$ is a Hilbert modular variety~\ref{conj-shimura} was previously shown by one of us (S.R.) \cite{reppen1}. The methods there also prove~\eqref{eq-ogus-principle} for $(G,\mu,r)$ if $\type(G)=\Asf_1^d$, $\mu$ is regular and $r$ is faithful of dimension $2d$. 
\begin{remark}
In upcoming joint work with Jean-Stefan Koskivirta we complete the computations in the unitary case to the case of inert primes and arbitrary signatures. As applications, we also use the order of vanishing to deduce (non)smoothness of Bruhat strata closures. We intend to obtain similar results for Ekedahl-Oort strata closures. 
\end{remark}
\subsection{Ogus' Principle for  abelian-type Shimura varieties and the moduli space of K3 surfaces}
Assume that $\GG$ is a $\QQ$-group such that $\GG_{\RR} \cong \SO(n-2,2), \GSO(n-2,2)$ or $\U(n-1,1)$. In the latter case assume that $G$ is split over $\FF_p$. Let $\gx$ be an abelian-type Shimura datum and $S_K$ as in~\ref{sec-intro-Hodge-type}. Let $r\colon G\to \GL(V)$ be either the standard representation $r=\std\colon G\to \GL(m)$ in the orthgonal case, or $r=\wedge^n(\std\oplus \std^{\vee})$ in the unitary case. 
Applying cases~\ref{item-th-main-odd-orthogonal-Bn},~\ref{item-th-main-even-orthogonal-Dn} and \ref{item-GL} of the main result~\ref{th-main} gives an Ogus Principle for the Shimura variety $S_K$:
\begin{theorem}
\label{th-intro-orthogonal}
Ogus' Principle holds for $(S_K,\zeta,r)$: For every $x \in S_K(k)$,
$$\ord_x \zeta^*\ha(G,\mu,r)=\clp_x(G,\mu,r).$$ \end{theorem}
In contrast to the Hodge-type case,~\ref{th-intro-orthogonal} illustrates that Ogus' Principle continues to hold in a case where $r$ has more than two $\mu$-weights; equivalently the Hodge filtration admits more than two graded pieces.

Specialize further to $m=21$.   Using Madapusi-Pera's extension of the Kuga-Satake construction to mixed characteristic \cite{Madapusi-Tate-K3} recovers Ogus' results for K3 surfaces within the $G$-Zip Geometricity framework.
\begin{corollary}
\label{cor-k3-intro}
Let $M^{\circ}_{2d, K}$ be the moduli space of K3 surfaces with level $K$ and a polarization of degree $2d$. Let $Y/M^{\circ}_{2d, K}$ be the universal family. Then Ogus' Principle holds for the CY $F$-Zip $H^2_{\dR}(Y/M^{\circ}_{2d, K})$.  
\end{corollary}
\subsection{Connection to the Bruhat stratification}
\label{sec-intro-Bruhat} Wedhorn \cite{wedhorn.bruhat}\cite{wedhorn.bruhat2} defined a Bruhat stratification of $\GZip^{\mu}$ and -- by pullback along $\zeta$ -- of the Shimura varieties $S_K$. The conjugate line position agrees with the pullback to $\GZip^{\mu}$ of a coarser Bruhat stratification of $\GLnZip^{r \circ \mu}$, where one replaces the `conjugate parabolic' $Q$ with the potentially larger parabolic of $\GL(n)$ stabilizing the conjugate line. 
For the triples $(G,\mu,r)$ treated in the main result~\ref{th-main}, we show that the pullback of this coarser Bruhat stratification of $\GLnZip^{r \circ \mu}$ agrees with true Bruhat stratification of $\GZip^{\mu}$.    
By definition, the conjugate line position is a special case of the relative position of two flags, where one flag is simply the conjugate line. In the Hodge-type case~(\S\ref{sec-intro-app-Shimura}), the conjugate line position is equivalent to the relative position of the flags $\fil^1_{\Hdg}$ and $\fil_1^{\Conj}$ in $H^1_{\dR}(Y/S_K)$. In the orthogonal abelian case~\ref{th-intro-orthogonal}, the conjugate line position determines the full relative position of the Hodge and conjugate filtrations of the universal $F$-Zip $\underline{\Vscr}(G,\mu,r)$ because the orthogonal $F$-Zips are self-dual, so the line of the conjugate filtration determines its hyperplane.  
\subsection{Approaches to 
Ogus' Principle}
\label{sec-intro-approaches}
The main result~\ref{th-main} is proved by separately computing the vanishing order $\ord_x\ha(G,\mu,r)$ and the conjugate line position $\clp_x(G,\mu,r)$ and observing -- oh miracle -- that they are equal. It is desirable to have a more conceptual explanation for why the two invariants are equal. As explained above, the consequences for Shimura varieties~\ref{cor-intro-Hodge-type},~\ref{th-intro-orthogonal} and moduli of K3 surfaces~\ref{cor-k3-intro} are deduced by pullback along a smooth Zip period map.
\subsubsection{Computing the vanishing order}
Several different methods with overlapping scopes are used to compute the Hasse invariant's vanishing order. In the Siegel case, an explicit description of the Hasse invariant on $\GspZip^{\mu}$ is used to compute the order directly. The other methods are all based on the diagram  
\begin{equation}
\label{eq-zipflag-sbt-diagram}  \xymatrixrowsep{1pc} \xymatrix{ & \GF^{\mu} \ar[rd]^-{\piflags} \ar[ld]_{\piflagz} & \\ \GZip^{\mu} & & \sbt:=[B \backslash G/B],  }
\end{equation}
introduced in \cite[2.1.2,~(2.3.1)]{goldring.koskivirta.invent}, where $B$ is an $\fp$-Borel subgroup of the parabolic $P$ of non-positive $\mu$-weights.

There is a character $\lambda \in X^*(B)$ and a highest weight section $f_{\lambda}$ on $\sbt$ such that 
$\piflagz^*\ha(G,\mu,r)=\piflags^*(w_0f_{\lambda})$, where $w_0$ is the longest element of the Weyl group $W$ of $G$ associated to a maximal torus $T\subset B$. 
Since the maps~\eqref{eq-zipflag-sbt-diagram} are smooth, one reduces to computing the order of $f_{\lambda}$. 
To this end, a general approach is: First, view $f_{\lambda}$ as a highest weight vector in the representation $H^0(G/B,\call(\lambda))$.  Second, explicitly describe such vectors on translates of Schubert cells of $G/B$ of the form $wU^{+}B/B$, where $U^{+}$ is the unipotent radical of the opposite Borel (\S\ref{section:method of jantzen/demazure}). We accomplish this when $w$ is a Weyl group element of ``simple enough'' form.  
In particular, our method covers the entire poset $\iw$ of minimal representatives of cosets of the Weyl group  $W_L$ of $L$ in 
$W$ 
in all cases of~\ref{th-main} except the Siegel case. 
These descriptions could be of independent interest. 

For~\ref{th-main}\ref{item-GL} ($G=\GL(n)$), a more elementary and direct way to compute the order of  $f_{\lambda}$ is given by studying the Pl\"ucker embedding of $G/B$.
\subsubsection{Computing the conjugate line position}
This is achieved using the "standard" $F$-Zips of Moonen-Wedhorn \cite[(1.9)]{moonen.wedhorn} which give an explicit combinatorial model of every isomorphism class of $F$-Zips over $k$.
\subsubsection{Functoriality of Ogus' Principle}\label{section: functoriality of ogus principle, introduction}
Let $\varphi\colon G\to G'$ be a morphism with central kernel between connected, reductive $\fp$-groups with isomorphic adjoint groups such that $\varphi(G)$ is normal in $G'$. 
Let $r'\colon G'\to \GL(V)$ be an $\fp$-representation. By definition,  $\clp_x(G,\mu, r' \circ \varphi)=\clp_{\varphi(x)}(G', \varphi \circ \mu,r')$ for all $x \in \gzip^{\mu}(k)$. If ${}^{I}W$ has simple enough form, then $\ord_x\ha(G,r,\mu)=\ord_{\varphi(x)}\ha(G',\varphi\circ \mu,r')$ (see \S\ref{section: functoriality of ogus principle, lemma section} for a precise statement). We expect this to hold in general. It holds if $\varphi$ is smooth, in particular if $\ker \varphi$ is a torus. 
In these cases, Ogus' Principle is equivalent for $(G,\mu,r)$ and $(G',\varphi \circ \mu, r')$. For example, the coincidental isomorphism $\Asf_3 \cong \Dsf_3$ implies that~\ref{th-main}\ref{item-th-main-GL(4)} is equivalent to~\ref{th-main}\ref{item-th-main-even-orthogonal-Dn} for $m=3$. 
\subsection{de Rham vs crystalline methods; \texorpdfstring{$F$}{F}-Zips vs \texorpdfstring{$F$}{F}-crystals}
\label{sec-intro-dR-vs-crys}
In both \cite{ogus-Calabi-Yau} and \cite{Ogus-height-strata-K3}, Ogus uses crystalline cohomology and related methods. To study the Hasse invariant and conjugate line position of a family $Y/X$, Ogus uses the $F$-crystal $H^n_{\crys}(Y/X)$. 
By contrast, the philosophy of $G$-Zip geometricity in general and this paper in particular only requires the de Rham cohomology $H^n_{\dR}(Y/X)=H^n_{\crys}(Y/X) \otimes_{W(k)} k$, \ie the reduction mod $p$ of the crystalline cohomology. It seems natural that there should be an explanation of Ogus' Principle and its generalizations considered here only using de Rham cohomology because both sides of the principle, the Hasse invariant and conjugate line position, are defined purely in terms of de Rham cohomology (or more generally $F$-Zips). This paper largely achieves this expectation.

An additional advantage of using de Rham cohomology and $F$-Zips instead of crystalline cohomology and $F$-crystals is that there is much less data to keep track of mod $p$. For example, there are finitely many isomorphism classes of $F$-Zips over $k$ of fixed rank, but for many such isomorphism classes, there are infinitely many isomorphism classes of $F$-crystals reducing to a single isomorphism class of $F$-Zips. Regarding the construction of generalized Hasse invariants when the classical one is identically zero, \cite{Goldring-Nicole-mu-Hasse} used the Dieudonn\'e crystal $H^1_{\crys}(Y/X)$ of a family of abelian schemes $Y/X$ to construct a $\mu$-ordinary Hasse invariant for unitary Shimura varieties, but the approaches in \cite{goldring.koskivirta.invent} and \cite{Koskivirta-Wedhorn-Hasse} based on $G$-Zips seem much more general and robust. For example, the divisors of the Hasse invariants constructed in \cite{goldring.koskivirta.invent} and \cite{Koskivirta-Wedhorn-Hasse} are given explicitly in terms of root pairings, while it is unclear to us how to determine the divisor of the Hasse invariant constructed via crystalline cohomology in \cite{Goldring-Nicole-mu-Hasse}.    
\subsection{Outline}
\label{sec-outline}
\S\ref{background} provides notation and background on reductive groups, Weyl groups and $G$-Zips. The technical heart of the paper is \S\ref{section:method of jantzen/demazure}: It carries out the approach sketched in~\ref{sec-intro-approaches} by obtaining explicit expressions for highest weight sections on the Schubert stack. 
The remaining sections apply this to prove the different cases of~\ref{th-main} and its applications to Shimura varieties and the moduli space of K3 surfaces: \S\ref{sec-orthogonal-case}, \S\ref{unitary.case}, \S\ref{siegel.case} treat the cases that $G$ is an orthogonal group, $G=\GL(n)$  and  $G=\GSp(2g)$ respectively.

\section*{Acknowledgements}
We thank Simon Cooper and Jean-Stefan Koskivirta for enlightening conversations. W.G. thanks Marc-Hubert Nicole for bringing Ogus' works on Hasse invariants to his attention.

Part of this work was carried out while S.R. was a JSPS International Research Fellow. W.G. thanks the Knut \& Alice Wallenberg Foundation for its support under Wallenberg Academy Fellow grant KAW 2019.0256 and grants KAW 2018.0356, KAW 2022.0308. W.G. thanks the Swedish Research Council for its support under grant \"AR-NT-2020-04924.

\section{Notation and background}\label{background}
\subsection{Generalities}
Throughout, let $p$ be a prime and let $k$ be an algebraic closure of $\fp$. Let $\sigma: k \to k$ denote the arithmetic Frobenius $a \mapsto a^p$. Given a $k$-scheme $X$, let $X^{(p)}:=X \otimes_{k,\sigma} k$ be its Frobenius twist. 
The unipotent radical of an algebraic $k$-group $H$ is denoted $R_u(H)$. Given subgroups $P,Q$ of $H$ and $E$ of $P \times Q$, the action of $E$ on $H$ is always given by $(a,b) \cdot h=ahb^{-1}$.  

\subsection{Root data and Weyl groups} 
Let $G$ be a connected, reductive $\fp$-group. Let $T \subset G$ be a maximal torus. 
\subsubsection{Root data} Let $(X^*(T), R; X_*(T), R^{\vee})$ be the root datum of $(G_k,T_k)$, where $X^*(T)$ (resp. $X_*(T)$) is the character (resp. cocharacter) group of $T_k$ and $R$ (resp. $R^{\vee}$) is the  set of roots (resp. coroots) of $T_k$ in $G_k$. For each $\alpha\in R$ let $U_\alpha$ denote the root group and $x_\alpha\colon \mg_a\to U_\alpha$ the root homomorphism. 
Let $R^{+}\subseteq R$ be the system of positive roots such that $\alpha \in R^+$  if and only if 
$U_{-\alpha}\subseteq B$. Let $\Delta \subset R^+$ denote the base of simple roots. Let $X^*_+(T)$ be the cone of $\Delta$-dominant characters.
\subsubsection{Weyl groups}
For $\alpha \in R$, let $s_{\alpha} \in \Aut(X^*(T)_{\QQ})$ be the root reflection $s_{\alpha}(x)=x-\langle x, \alpha^{\vee} \rangle \alpha$.   Let $W=W(G,T)$ be the Weyl group of $G_k$ relative $T_k$. The pair $(W, \{s_{\alpha}\}_{\alpha \in \Delta})$ is a Coxeter system. Write $l:W \to \NN$ for its length function and $\leq$ for its Bruhat-Chevalley order. By \cite[Exp. XXIII, Section 6]{sga3}, every $w\in W$ admits a lift $\Dot{w}\in N_G(T)$ such that $\Dot{(w_1w_2)}=\Dot{w}_1\Dot{w}_2$ whenever $l(w_1w_2)=l(w_1)+l(w_2)$. 
By abuse of notation we omit the $\cdot$ and denote $\Dot{w}$ simply by $w$; it should be clear from context whether $w$ is viewed as an element of $W$ or of $N_G(T)$. 

For $I \subset \Delta$, let $W_I=\langle s_{\alpha} | \alpha \in I\rangle$ be the associated standard parabolic subgroup of $W$ and $w_{0,I}$ its longest element. 
Write $w_0:=w_{0,\Delta}$ for the longest element of $W$. 
For $I,J\subset \Delta$, 
let ${}^{I}W$ (resp. $W^J$, ${}^{I}W^{J}$) denote the set of minimal length representatives of the cosets (resp. double cosets) $W_I \backslash W$ (resp. $W/W_J$, $W_I \backslash W /W_J$). 
\subsubsection{The Hodge character}
\label{sec-hodge-char}
Let $\mu \in X_*(G)$. Let $\rho:G \to \GL(V)$ be a representation with precisely two $\mu$-weights $a<b$. Let $L:=\cent_{G_k}(\mu)$ be the Levi subgroup centralizing $\mu$. The Hodge character of $(G,\mu,\rho)$ is $\eta=\eta(G,\mu,\rho)=\det V_{a}$, where $V_a$ is the $a$-weight space for $\mu$. One has $\eta \in X^*(L)$. The exterior power $\wedge^{\dim V_b}r$ is of CY-type.  

\subsection{\texorpdfstring{$G$}{G}-zips}
\label{background.gzips}
Let $\mu \in X_*(G)$ be a cocharacter. It determines a pair of opposite parabolics $P, P^{+}$ of $G_k$ intersecting in a common Levi factor $L:=P \cap P^+=\cent_{G_k}\mu$. The Lie algebra of $P$ (resp. $P^+$) is the sum of the non-positive (resp. non-negative) $\mu$-weight spaces in $\fkg$. Let $Q\coloneqq (P^+)^{(p)}$ and $M\coloneqq L^{(p)}$. As in \cite[1.2.3]{goldring.koskivirta.invent},  assume that there exists an $\fp$-Borel subgroup $B \subset P$. Let $I$ (resp. $J$) denote the type of $P$ (resp. $Q$). Let $z\coloneqq w_0w_{0,J}$, the longest element in $W^J$.
%
%
%
%
%
 
%
%
%
%
%
%
%
%
%
%
\subsubsection{The stack of \texorpdfstring{$G$-zips and $G$-zip flags}{gzip}}
We briefly recall the various notions of zips that we use, and we refer to the original papers \cite{pink.wedhorn.ziegler.zip.data, pink.wedhorn.ziegler.additional, goldring.koskivirta.flag} for details. 
For any $k$-scheme $S$, we have the following definitions.

A \textit{$G$-zip of type $\mu$} over $S$ is a tuple $(I,I_P,I_Q, \iota)$ where $I$ is a $G_k$-torsor over $S$, $I_P\subset I$ is a $P$-torsor over $S$, $I_Q \subset I^{(p)}$ is a $Q$-torsor over $S$ and 
\begin{equation*}
    \iota : I_{P}^{(p)}/R_u(P)^{(p)} \xrightarrow{\cong} I_Q/R_u(Q)
\end{equation*}
is an isomorphism of $M$-torsors. A morphism of $G$-zips $(I,I_P,I_Q,\iota)\to (I',I_P',I_Q',\iota')$ is a morphism of $G_k$-torsors $I\to I'$ compatible with the reductions to $P$ and $Q$ and compatible with the morphisms $\iota, \iota'$.

A \textit{$G$-zip flag of type $\mu$} over $S$ is a pair $(\underline{I},I_B)$ where $\underline{I}=(I,I_P,I_Q,\iota)$ is a $G$-zip of type $\mu$ over $S$ and $I_B\subset I_P$ is a $B$-torsor. A morphism of $G$-zip flags is a morphism of the underlying $G$-zips compatible with the reductions to $B$.

Let $\gzip^{\mu}(S)$ (resp. $\gzipf^{\mu}(S)$) denote the category of $G$-zips (resp. $G$-zip flags) of type $\mu$ over $S$. Both of these constructions give rise to smooth algebraic $k$-stacks $\gzip^{\mu}$ resp. $\gzipf^{\mu}$. 
The forgetful map $(\underline{I},I_B)\mapsto \underline{I}$ induces a smooth morphism
\[
\piflagz\colon \gzipf^{\mu}\to \gzip^{\mu}.
\]
\subsubsection{Description as quotient stacks}
Let $E\coloneqq P\times_M Q$, and let $E'\coloneqq E\cap (B_k\times G_k)$. The groups $E$ and $E'$ act on $G_k$ via $(a,b)\cdot g = agb^{-1}$. By \cite[Theorem 2.1.2.]{goldring.koskivirta.invent} the natural projections induce a commutative diagram
\begin{equation*}
    \begin{tikzcd}
       \gzipf^{\mu} \arrow[r, "\piflagz"] \arrow[d, "\cong"] & \gzip^{\mu} \arrow[d, "\cong"] \\
    \left[E'\backslash G_k \right] \arrow[r] & \left[ E\backslash G_k \right] 
    \end{tikzcd}
\end{equation*}
%
\subsubsection{}
The \textit{Schubert stack} is the quotient stack 
\begin{equation*}
    \sbt \coloneqq [(B_k\times B_k)\backslash G_k].
\end{equation*}
The map $g \mapsto gz$ induces an isomorphism $B_k\times B_k\cong B_k\times {}^{z}B_k$, whence an isomorphism $\sbt \cong [(B_k\times {}^{z}B_k)\backslash G_k]$. Composing with the natural projection $\gzipf^{\mu}\cong [E'\backslash G_k] \to [(B_k\times {}^{z}B_k)\backslash G_k]$ gives a smooth morphism
\begin{equation*}
    \piflags : \gzipf^{\mu} \to \sbt.
\end{equation*}
\subsubsection{}
The \textit{Bruhat stack} is the quotient stack
\begin{equation*}
    \mathcal{B}_{I,J,\Delta} \coloneqq [P\backslash G_k/Q],
\end{equation*}
By \cite[Section 2]{wedhorn.bruhat}, the stack $\mathcal{B}_{I,J,\Delta}$ is smooth and the morphism $\pi_{\mathtt{Zip},\mathcal{B}}\colon \gzip^{\mu}\to \mathcal{B}_{I,J,\Delta}$ induced by the inclusion $E\to P\times Q$ is representable, of finite type and smooth. 
%
\subsubsection{Stratifications}\label{section:stratifications}
The Bruhat stratification expresses $G$ as the disjoint union of locally closed subschemes
\begin{equation*}\label{g.is.union.of.b}
    G = \coprod_{w\in W} BwB.
\end{equation*}
Define $\sbt_w \coloneqq [(B_k\times B_k)\backslash B_kwB_k]$ and $\mathcal{X}_w \coloneqq \piflags^{-1}(\sbt_w)$. The substacks $\sbt_w$ (resp. $\mathcal{X}_w$) form a stratification of $\sbt$ (resp. $G\zipf^{\mu}$) into locally closed substacks. Similarly there is a stratification of $\mathcal{B}_{I,J,\Delta}$ indexed by ${}^{I}W^{J}$.

Under the action of $E$ on $G_k$ the $E$-orbits form a stratification of $G_k$ which then induces a stratification of $\gzip^{\mu}$. There is an explicit bijection \cite[(1.4.2)]{goldring.koskivirta.invent} 
\begin{equation*}
    \begin{aligned}
        {}^{I}W&\to \{ \text{$E$-orbits in $G$} \} 
    \end{aligned}
\end{equation*}
Denote by $G_w$ the stratum corresponding to $w \in {}^{I}W$. By \cite[Proposition 2.4.3]{goldring.koskivirta.invent},  $\piflagz(\overline{\calx}_w)=[E\backslash \overline{G}_w]$ and $\piflagz(\calx_w)=[E\backslash G_w]$ for all $w \in \iw$.
\subsubsection{Zip period maps}
Consider a morphism of $k$-stacks $\zeta:X \to \gzip^{\mu}$, referred to as a \textit{Zip period map} (including the case $X=\gzip^{\mu}$, $\zeta=\id$), and let $w\in {}^{I}W$. We refer to $X_{w}^{\EO}\coloneqq \zeta^{-1}(w)$ as the \textit{zip stratum}, or the \textit{Ekedahl-Oort stratum}, of $w$. Similarly if $w\in {}^{I}W^{J}$ corresponds to a point $x_w\in \mathcal{B}_{I,J,\Delta}$ then we refer to $X_w^{\bruh}\coloneqq (\pi_{\mathtt{Zip},\mathcal{B}}\circ \zeta)^{-1}(w)$ as the \textit{Bruhat stratum} of $w$.
%
%
\subsubsection{Associated line bundles}
\label{line.bundles.on.stacks}
Let $\nu \in X^*(L)$ and $\lambda, \lambda' \in X^*(T)$. The associated sheaves construction \cite[Section I.5.8]{jantzen} gives line bundles $\mathcal{L}_{\gzip}(\nu)$, $\mathcal{L}_{\gzipf}(\lambda)$, $\mathcal{L}_{\sbt}(\lambda, \lambda')$ and $\mathcal{L}_{G/B}(\lambda)$ on $\gzip^{\mu}$‚ $\gzipf^{\mu}$, $\sbt$ and $G/B$ respectively. Pullback along $\piflagz$ and $\piflags$ yields line bundles on $\gzipf^{\mu}$. 
By \cite[Lemma 3.1.1]{goldring.koskivirta.invent} we have that $\piflagz^*\Lcal_{\gzip}(\nu)= \Lcal_{\gzipf}(\nu)$ and
\begin{equation*}
\begin{aligned}
    \piflags^{*}\mathcal{L}_{\sbt}(\lambda, \lambda') &= \mathcal{L}_{\gzipf}(\lambda + p{}^{\sigma^{-1}}(z\lambda')).
\end{aligned}
\end{equation*}
By \cite[Theorem 2.2.1]{goldring.koskivirta.invent} the space $H^{0}(\sbt, \mathcal{L}_{\sbt}(\lambda, \lambda'))\neq 0$ if and only if $\lambda\in X_{+}^*(T)$ and $\lambda' = -w_0\lambda$. Define 
\begin{equation*}
 D_{w_0} : X^{*}(T)\to X^{*}(T) \quad \textnormal{ by } \quad  D_{w_0}(\lambda) := \lambda - p{}^{\sigma^{-1}}(zw_{0}^{-1}\lambda).
\end{equation*}
Then 
\begin{equation*}
    \piflags^{*}\mathcal{L}_{\sbt}(\lambda , - w_0 \lambda) = \mathcal{L}_{\gzipf}(D_{w_0}(\lambda)).
\end{equation*}
Let $\lambda\in X_+^*(T)$. Let $H^0(\lambda)\coloneqq H^0(G/B,\Lcal_{G/B}(\lambda))$. By \cite[Remark 5.2]{reppen1}) we have that
\begin{equation*}
    H^{0}(\sbt,\call(-w_0\lambda,\lambda))= H^{0}(\lambda)_{w_0\lambda}.
\end{equation*}
\subsubsection{The Hodge line bundle}
\label{sec-hodge-line-bundle}
The line bundle $\Lcal_{\GZip}(\eta)$ associated to the Hodge character~\ref{sec-hodge-char} is the Hodge line bundle. In the Hodge-type Shimura variety setting~\ref{sec-intro-Hodge-type}, let $\varphi$ be a symplectic embedding~\eqref{eq-symplectic-embedding} and let $\rho$ be the dual of the induced embedding $G \to \GSp(2g) \to \GL(2g)$. Then the two $\mu$-weights of $\rho$ are $0,-1$. The Hodge line bundle of $\GZip^{\mu}$ pulls back to the Hodge line bundle of the Shimura variety $S_K$ via the Zip period map~\eqref{eq-intro-zip-period}. 
\subsection{Order of vanishing}\label{section.order.of.vanishing}
Let $S$ be a smooth scheme over $k$, $\mathcal{L}$ a line bundle on $S$ and $f \in H^0(S, \mathcal{L})$ a global section. Denote by $\mathcal{I}$ the image of the morphism $\mathcal{L}^{-1}\to \mathcal{O}_S$ induced by $f$. The zero scheme of $f$ is defined to be the subscheme determined by $\mathcal{I}$ (cf. 
\cite[Appendix A, C.6]{hartshorne}). 
For all $n \in \NN$, we say that $f$ vanishes to order $n$ at a closed point $s$ in $S$ if $\mathcal{I}_s$ is contained in $\mathfrak{m}_s^n$, where $\fkm_s$ is the maximal ideal of $\mathcal{O}_{S,s}$. The order of vanishing of $f$ at $s$ is defined to be the maximal such integer $n$, denoted by $\ord_s(f)$. Let $[X/H]$ be any of the stacks of \S\ref{background.gzips}, let $\chi \in X^*(H)$ be a character and let $f \in H^0(\chi)$. 
The order of vanishing of $f$ on $\overline{x}\in [X/H](k)$ is defined to be the order of vanishing of $f$ on $x\in X(k)$, where $x$ is a lift $\overline{x}$. This is well-defined since $f(xh)=\chi^{-1}(h)f(x)$ for all $x$ in $X$ and $h$ in $H$. Finally, for any scheme $S$ and smooth morphism $\zeta\colon S\to [X/H]$, we have that $\ord_s(\zeta^*f)=\ord_{\zeta(s)}(f)$.
\section{Explicit expression of highest weight sections on Schubert cells}\label{section:method of jantzen/demazure}
%
\subsection{Geometry of reductive groups}\label{section: geometry of reductive groups}
Fix a connected reductive group $G$ over an algebraically closed field $k$ of characteristic $p>0$. We fix also a maximal Borus $T\subset B$, and we let $B^{+}$ denote the opposite Borel and $U$ and $U^{+}$ denote the unipotent part of $B$ respectively $B^{+}$. Let $R$ be a system of roots corresponding to $T$, let $\Delta$ be a system of simple roots determined by $B^{+}$ and let $R^{+}$ denote the corresponding positive roots. Let $W=W(G,T)$ denote the Weyl group. For a root $\alpha \in R$, let $U_{\alpha}$ denote the root group corresponding to $\alpha$ and $x_{\alpha} : \mg_{a}\to G$ the corresponding root group morphism.

Recall that a subset $R'$ of $R^+$ is \textit{closed} if for all $\alpha, \beta\in R'$, $(\NN\alpha + \NN\beta)\cap R\subset R'$. The subset $R'$ is called \textit{unipotent} if $R'\cap(-R')=\emptyset$. For a closed and unipotent subset $R'\subset R$, let let $U(R')$ be the subgroup of $G$ generated by all $U_{\alpha}$, with $\alpha \in R'$. For such a subset, multiplication induces an isomorphism $\prod_{\alpha\in R'}U_{\alpha}\cong U(R')$, where the product is taken in any order (cf. Jantzen \cite[Section II.1.7.]{jantzen}). 

Given $w\in W$ let $R'(w)\coloneqq \{ \alpha \in R^{+} : w^{-1}(\alpha) > 0 \}$ and $R''(w) \coloneqq \{ \alpha \in R^{+} : w^{-1}(\alpha)< 0 \}$. 
%
As explained in \cite[Section II.1.9]{jantzen}, the root group morphisms and group multiplication give a commutative diagram of isomorphisms
\begin{equation*}
\begin{tikzcd}
    k^{|R'(w)|} \times k^{|R''(w)|} \arrow[d] \arrow[r] & k^{|R^{+}|} \arrow[d] \\
    \prod_{\alpha \in R'(w)} U_{\alpha} \times \prod_{\alpha \in R''(w)} U_\alpha \arrow[d] \arrow[r] & \prod_{\alpha \in R^{+}} U_\alpha \arrow[d] \\
    U(R'(w))\times U(R''(w)) \arrow[r] & U^{+}
\end{tikzcd}.
\end{equation*}
The (composition of the) vertical morphisms induces thus the diagram
\begin{equation*}
\begin{tikzcd}
    B^{+}wB = wU(w^{-1}R'(w))B \arrow[r, hookrightarrow] \arrow[d, "\cong"] & wU^{+}B \arrow[d, "\cong"] \\
    k^{|R^{+}|-l(w)}\times B \arrow[r, hookrightarrow] & k^{|R^{+}|}\times B
\end{tikzcd}
\end{equation*}
from which we can identify $B^{+}wB$ with the subset
\begin{equation*}
\begin{aligned}
    \Big\{ (a_1, ..., a_{|R^{+}|})\in k^{|R^{+}|} : a_{|R^{+}|-l(w)+1}= ... = a_{|R^{+}|}=0 \Big\} \times B \\
    \subset k^{|R^{+}|}\times B \cong wU^{+}B.
\end{aligned}
\end{equation*}
Given a function $f$ on $wU^{+}B$, we obtain the order of vanishing at points in $B^{+}wB$ under this identification.
\begin{remark}\label{general.strategy.f.to.w0f}
If $f$ is a function on $wU^{+}B$, then $w_0 \cdot f$ is a function on $w_{0}wU^{+}B$ and for all $g \in wU^{+}(k)B(k)$, we have that $\text{ord}_{w_{0}g}(w_0 \cdot f) = \text{ord}_{g}(f)$. 
Since 
\begin{equation*}
    Bw_0 w B = w_0 B^{+}wB = w_0 wU(w^{-1}R'(w))B \subset w_0wU^{+}B,
\end{equation*}
we see that the vanishing order of $w_0 \cdot f$ on $B w_0 w B$ equals that of $f$ on $B^{+}w B$.
\end{remark}
\subsection{Order of vanishing via highest weight vectors}\label{general.strategy}
From \ref{line.bundles.on.stacks} and \S\ref{section: geometry of reductive groups} we find the following general strategy to compute the order of vanishing of the Hasse invariant. 
Let $r$ be a representation of $G$ of CY-type such that $\ha(G,\mu,r)$ is not identically zero. Let $\eta$ be the negative of the highest weight of $(r,V)$. Then $\ha(G,\mu,r)\in H^0(\gzip^\mu, \Lcal_{\gzip}((p-1)\eta))$. Let $\lambda$ be a character such that $D_{w_0}(\lambda) = (p-1)\eta$. Then $\piflagz^*\ha(G,\mu,r)$ is nonvanishing on $\calx_{w_0}$ and thus there is a section $\ha_{\sbt}\in H^0(\sbt, \call_{\sbt}(\lambda,-w_0\lambda))$ such that $\piflags^*\ha_{\sbt}=\piflagz^*\ha(G,\mu,r)$ (cf. \cite[Proposition 5.1.3]{koskivirta.imai}). Suppose that $w_0\lambda =-\lambda$. Then 
\begin{equation}
\begin{aligned}
    \piflags^*\call_{\sbt}(\lambda,-w_0\lambda)&
    = \piflagz^*\Lcal_{\gzip}((p-1)\eta)\\
    H^0(\sbt, \call_{\sbt}(\lambda,-w_0\lambda)) &= H^0(\lambda)_{w_0\lambda}.
\end{aligned}
\end{equation}
Let $f_\lambda \coloneqq w_0\cdot \ha_{\sbt}\in H^0(\lambda)_{\lambda}$. For any $w\in {}^{I}W$, by \ref{general.strategy.f.to.w0f} we see that to obtain the order of vanishing of $\ha_{\sbt}$ on points in $Bw_0 wB$ it suffices to obtain the order of vanishing of $f_{\lambda}$ on points in $B^+ w B$. Since the map $\piflags\colon \gzipf^{\mu}\to \sbt$ is induced by the morphism $G\to G$, $g\mapsto gz$, and since $\piflagz^*\ha(G,\mu,r) = \piflags^*\ha_{\sbt}$, we 
thus see that to obtain the order of vanishing of $\ha(G,\mu,r)$ on the Ekedahl-Oort stratum in which $w_0wz^{-1}$ lies, it suffices to obtain the order of vanishing of $f_{\lambda}$ on all points in $B^+ w B$ (cf. \ref{section:stratifications}).
%
%
\subsection{Global sections of highest weight modules}
Let $\lambda\in X^*_+(T)$ and $f_\lambda\in H^0(\lambda)_{\lambda}$. 
The classical construction of $f_{\lambda}$ as a function on $G$ first gives an explicit description of $f_{\lambda}$ on 
\begin{equation}
\label{eq-flambda-codim-1}
U^{+}B\cup \bigcup_{\alpha \in \Delta}s_{\alpha}U^{+}B.    
\end{equation}
 It then uses that the codimension of~\eqref{eq-flambda-codim-1} in $G$ is at least 2 coupled with the normality of $G$ to deduce that $f_{\lambda}$ extends to all of $G$.

Motivated by \S\ref{general.strategy} we extend the explicit description of $f_\lambda$ to subsets $wU^+B\subset G$ for Weyl group elements $w$ of ``simple enough'' form (when written as a product of simple reflections). 
The idea for this construction is the following. 
If $w = \prod_\alpha s_\alpha$ is a product of simple reflections $s_\alpha$, then we use that $s_\alpha$ normalizes $\prod_{\beta\in R\setminus \{\alpha\}} U_\beta$ to successively rewrite $wU^{+}$ as $\Big(\prod_{\beta\in R^{+}\setminus I_w}U_\beta\Big)\Big(\prod_\alpha s_\alpha U_\alpha\Big)$, for some suitably chosen indexing set $I_w$. 
We then use some facts about the root theory of reductive groups to manipulate an expression of the form $u^+ \prod_\alpha s_\alpha x_{\alpha}u_0$, with $u^{+}\in U^+$ and $u\in U$, into being of the form $u_0^{+}tu_0$, with $u_0^{+}\in U^+$, $t\in T$ and $u_0\in U$. Since we know from the classical construction that $f_\lambda(u_0^{+}tu_0) = \lambda^{-1}(t)$, we obtain an explicit description of $f_\lambda$ on $wU^+B$ and thus the order of vanishing of $f_\lambda$ on points therein. 

\subsection{Notation}
\label{sec-notation-root-computations}
We list here the facts used in the computations below (see \cite[Part II Chapter 1]{jantzen} for a reference). We also introduce some notation that will be convenient throughout the proofs.
%
%
%
%
%
%
\begin{enumerate}[label=\textbf{R.\arabic*}]
    \item
    \label{item-closed-root-set}
    For any set of positive roots $\gamma_1, ..., \gamma_m\in R^+$, let 
    \begin{equation*}
        U^{+,\gamma_1, ..., \gamma_n} = \prod_{\gamma\in R^+ \setminus \{ \gamma_1, ..., \gamma_n\}}U_\gamma.
    \end{equation*}
    If $w=\prod_{i=1}^n s_{\alpha_i}$ is a product of simple reflections, let
    \[
    R^{w}\coloneqq \{ \alpha_{1}, s_{\alpha_1}(\alpha_{2}), s_{\alpha_1}s_{\alpha_2}(\alpha_{3}),..., s_{\alpha_1}\cdots s_{\alpha_{n-1}}(\alpha_{n})\}
    \]
    and let
    \[
    R^{+,w}\coloneqq R^{+}\setminus R^{w}.
    \]
    \item\label{com.relations} For two roots $\gamma_1, \gamma_2\in R$ such that $\gamma_1 + \gamma_2 \neq 0$, there are $c_{ij}\in \ZZ$ for all $i,j>0$ with $i\gamma_1+j\gamma_2\in R$ and such that
    \begin{equation*}
    x_{\gamma_1}(a)x_{\gamma_2}(b) = \prod_{i,j>0}x_{i\gamma_1 + j\gamma_2}(c_{ij}a^{i}b^{j})x_{\gamma_2}(b)x_{\gamma_1}(a).
\end{equation*}
Given such $i,j$'s, we will write $x_{I\gamma_1 + J\gamma_2} = \prod_{i,j>0}x_{i\gamma_1 + j\gamma_2}$, for simplicity. Similarly, if $w$ is an element of the Weyl group, let $x_{w(I\gamma_1 + J\gamma_2)}\coloneqq \prod_{i,j>0}x_{w(i\gamma_1 + j\gamma_2)}$.
\item\label{wx} If $w$ is an element of the Weyl group, and $\gamma$ is a root, then 
\begin{equation*}
    w x_{\gamma}(a) = x_{w\gamma}(\pm a)w,
\end{equation*}
for all $a\in k$.
\item\label{tx} For any $t\in T$ and $\gamma \in R$‚ we have that
\begin{equation*}
    tx_{\gamma}(a)=x_{\gamma}(\gamma(t)a)t,
\end{equation*}
for all $a\in k$.
\item\label{sx} For any $\gamma \in R$, we have that
\begin{equation*}
    s_{\gamma}x_{\gamma}(a) = x_{\gamma}(-a^{-1})\gamma^{\vee}(-a^{-1})x_{-\gamma}(a^{-1}),
\end{equation*}
for all $a\in k$.
\end{enumerate}
%
%
%
\subsection{Closedness of \texorpdfstring{$R^{+,w}$}{} root sets}\label{assumption that subset of roots is closed}
We state a condition used in \S\ref{section: product of distinct simple root reflections}-\S\ref{section: computations for type Dn}: 
\begin{condition}
\label{cond-closed}
For a product of simple reflections $s_1 \cdots s_n$,  the set $R^{+,s_{j}s_{j+1}\cdots s_{n}}$ is closed for all $j>1$.      
\end{condition}

\subsubsection{Closedness for type \texorpdfstring{$\Bsf_m$}{Bn}}  Consider a $\Bsf_m$ root system, with simple roots $\alpha_i = e_{i}-e_{i+1}$ for $i=1,...,m-1$ and $\alpha_m=e_m$. Fix $l< m$, let $w_j\coloneqq s_j\cdots s_{m-1} s_m s_{m-1}\cdots s_{m-l}$ for all $j\leq m$. A computation shows that
\begin{equation*}
    \begin{aligned}
        R^{w_j} &= \{ \alpha_j, s_j(\alpha_{j+1}), s_js_{j+1}(\alpha_{j+2}), ..., s_j\cdots s_m s_{m-1}\cdots s_{m-l-1}(\alpha_{m-l}) \} \\
        &=\{ e_j - e_{j+1}, ..., e_j -e_m, e_j, e_j + e_m, e_j + e_{m-1}, ..., e_j + e_{m-l+1} \}
    \end{aligned}
\end{equation*}
We proceed to show that $R^{+,w_j}=R^+\setminus R^{w_j}$ is closed.

Suppose on the contrary that there exist $\alpha, \beta \in R^{+,w_j}$ and that there are $a,b\in \NN$ such that $a\alpha + b\beta \in R^{w_j}$. We divide into cases:
\begin{enumerate}
\item\label{item: case 1 closedness for bn} 
Suppose first that $a\alpha+b\beta = e_j - e_k$. Since $e_j - e_k=\sum_{i=j}^{k-1}\alpha_i$ and the simple roots are linearly independent, we must have that $a=b=1$ and there are disjoint subsets $I_\alpha,I_\beta\subset \{ j, ..., k-1\}$ such that $\alpha = \sum_{i\in I_\alpha}\alpha_i$ and $\beta = \sum_{i\in I_\beta}\alpha_i$. If there are some $i,i'\in I_\alpha$ such that $i+1<i'$ and $i'-1\notin I_\alpha$, then $e_{i'}$ has positive coefficient when we write $\alpha$ as a sum of the $e_i$'s. Similarly, if $i''\coloneqq \max\{r\in I_\alpha : r<i' \}$ then $e_{i''+1}$ has a negative coefficient when we write $\alpha$ as a sum of the $e_i$'s. Since $\alpha$ is a root, this implies that $\alpha = e_{i'} - e_{i''}$. But since $i''< i'$, this contradicts that $\alpha$ is a positive root. Thus, there is some $i'$ such that $I_\alpha = \{j,j+1,...,i'\}$, and consequently $\alpha = \sum_{i=j}^{i'}\alpha_i=e_j-e_{i'+1}$. This contradicts that $\alpha\notin R^{w_j}$.
\item\label{item: case 2 closedness for bn}
Suppose now that $a\alpha + b\beta = e_j + e_k$, for some $k\geq m-l+1$. Since $e_j \in R^{w_j}$ we may assume that $\alpha = e_j \pm e_r$ for some $r>j$.\footnote{Indeed, $e_j$ must occur with positive coefficient for either $\alpha$ or $\beta$ and since $\alpha,\beta\notin R^{w_j}$ we see that both of them must be of the form $e_r\pm e_s$ for some integers $s,r$. Thus, for either $\alpha$ or $\beta$, $s=j$.} If $\alpha = e_j - e_r$ then $\alpha \in R^{w_j}$ so we have that $\alpha = e_j + e_r$, and we must have that $r<m-l+1$ (otherwise $\alpha\in R^{w_j}$). Then we must have that $\beta = e_s - e_r$ for some $s$ such that $s<r$ and $s\neq j$ (else $\beta \in R^{w_j}$), and we must have that $a=b=1$. But then $e_j+e_k =a\alpha + b\beta = e_j + e_s$, and since $s<r<m-l+1$ and $k\geq l-m+1$ this is a contradiction. 
\item
Finally $a\alpha + b\beta \neq e_j$ since $e_j$ is a simple root.
\end{enumerate}
This shows that $R^{+,w_j}$ is closed.
\subsubsection{Closedness for type \texorpdfstring{$\Dsf_m$}{Dn}}
Consider the case of a $\Dsf_m$ root system, with simple roots $\alpha_i = e_{i}-e_{i+1}$ for $i=1,...,m-1$ and $\alpha_m=e_{m-1}+e_m$. Fix $l< m$, let $w_j\coloneqq s_j\cdots s_{m-1} s_{m} s_{m-2}\cdots s_{m-l}$ for all $j\leq m$. A computation shows that
\begin{equation*}
    \begin{aligned}
        R^{w_j} &= \{ \alpha_j, s_j(\alpha_{j+1}), s_js_{j+1}(\alpha_{j+2}), ..., s_j\cdots s_{m-1} s_{m} s_{m-2}\cdots s_{m-l-1}(\alpha_{m-l}) \} \\
        &= \{ e_j - e_{j+1}, ..., e_j - e_{m-1}, e_j + e_{m}, e_j + e_{m-1}, ..., e_j + e_{m-l+1} \}.
    \end{aligned}
\end{equation*}
We show that $R^{+,w_j}=R^{+}\setminus R^{w_j}$ is closed. Suppose on the contrary that there are $\alpha, \beta \in R^{+,w_j}$ such that $a\alpha +b\beta = e_j - e_k$ for some $a,b\in \NN$.

If $a\alpha+b\beta = e_j - e_k$ or if $a\alpha + b\beta = e_j + e_m$, then the same argument as in \ref{item: case 1 closedness for bn} in the $\Bsf_m$-case applies. If  $a\alpha + b\beta = e_j + e_k$ for some $k\geq m-l+1$, then the same argument as in \ref{item: case 2 closedness for bn} in the $\Bsf_m$-case applies. Thus, $R^{+,w_j}$ is closed.
\subsection{Products of distinct simple reflections}\label{section: product of distinct simple root reflections}
\begin{theorem}\label{distinct.simple}
Let $\lambda\in X^{*}_+(T)$ and  $f_{\lambda}\in H^{0}(\lambda)_\lambda$. Let $w=s_{\alpha_1}\cdots s_{\alpha_n}$ be  a product of distinct simple reflections satisfying~\ref{cond-closed}. The order of vanishing of $f_{\lambda}$ at every point in $B^{+}wB$ is  
    \begin{equation*}
        \ord(f_{\lambda}) = \sum_{i=1}^{n}\langle \lambda , \alpha_{i}^{\vee} \rangle.
    \end{equation*}
\end{theorem}
\begin{proof}
Let $w=s_{\alpha_1}\cdots s_{\alpha_n}$ for distinct simple root reflections $s_{\alpha_1},..., s_{\alpha_n}$.
%
Since $s_{\alpha}$ normalizes $U^{+,\alpha}$ for all simple roots $\alpha$, an induction on $n$ using~\ref{cond-closed} shows that
\begin{equation*}
    wU^+ = U^{+,\alpha_1, s_{\alpha_1}(\alpha_2), s_{\alpha_1}s_{\alpha_2}(\alpha_3), ..., s_{\alpha_1}\cdots s_{\alpha_{n-1}}(\alpha_n)}s_{\alpha_1}U_{\alpha_1}\cdots s_{\alpha_n}U_{\alpha_n}.
\end{equation*}
Hence we have an isomorphism of varieties
\begin{equation}\label{identification.gb}
    \begin{aligned}
    U^{+,\alpha_1, s_{\alpha_1}(\alpha_2), s_{\alpha_1}s_{\alpha_2}(\alpha_3), ..., s_{\alpha_1}\cdots s_{\alpha_{n-1}}(\alpha_n)} \times k^{n} \times T \times U &\to wU^{+}B \\
    (u_1, a_1, ..., a_n, t, u ) \mapsto u_1 \prod_{i=1}^{n}s_{\alpha_i}x_{\alpha_i}(a_i) tu.
    \end{aligned}
\end{equation}
Under this identification $B^{+}wB$ is the zero locus of the coordinates of $k^n$. Now we rewrite the image of this morphism to obtain an expression lying in $U^{+}TU$ in order to apply the explicit description of $f_\lambda$. We first prove by induction on $n$ that for all $(a_1, ..., a_n) \in k^n$, 
\begin{equation}\label{equation: coordinate change distinct simple reflections}
    \prod_{i=1}^{n}s_{\alpha_i}x_{\alpha_i}(a_i) = u^{+} \prod_{i=1}^{n}\alpha_{i}^{\vee}(-a_{i}^{-1}) u^{-},
\end{equation}
for some $u^{+}\in U^+$ and $u^{-}=\prod_{i=1}^{n}x_{-\alpha_i}(a_i') \in U$, for some $a_i'$. If $n=1$ we are done by 
(\ref{sx}). So suppose that~\eqref{equation: coordinate change distinct simple reflections} is true for $n-1$ for some $n\geq 2$. Then by induction and using \ref{sx} we have 
\begin{equation*}
    \prod_{i=1}^{n}s_{\alpha_i}x_{\alpha_i}(a_i) = u^{+} \prod_{i=1}^{n-1}\alpha_{i}^{\vee}(-a_{i}^{-1}) \prod_{i=1}^{n-1}x_{-\alpha_i}(a_i')x_{\alpha_n}(-a_{n}^{-1})\alpha_{n}^{\vee}(-a_{n}^{-1})x_{-\alpha_n}(a_{n}^{-1}),
\end{equation*}
for some $a_i' \in k$ and some $u^{+}\in U^{+}(k)$. 
Since $\alpha_i\neq \alpha_n$ for all $i\neq n$ we have $x_{-\alpha_i}(a)x_{\alpha_n}(b)=x_{\alpha_n}(b)x_{-\alpha_i}(a)$ for all $a,b \in k$. Thus
\begin{equation*}
    \prod_{i=1}^{n}s_{\alpha_i}x_{\alpha_i}(a_i) =u^{+} \prod_{i=1}^{n-1}\alpha_{i}^{\vee}(-a_{i}^{-1})x_{\alpha_n}(-a_{n}^{-1})\prod_{i=1}^{n-1}x_{-\alpha_i}(a_i')\alpha_{n}^{\vee}(-a_{n}^{-1})x_{-\alpha_n}(a_{n}^{-1}).
\end{equation*}
By \ref{tx}, 
\begin{multline*}
    \prod_{i=1}^{n}s_{\alpha_i}x_{\alpha_i}(a_i) =
    u^{+}x_{\alpha_n}\Big(-\prod_{i=1}^{n-1}\alpha_{n}(\alpha_{i}^{\vee}(-a_{i}^{-1})a_{n}^{-1}\Big) \cdot
    \prod_{i=1}^{n}\alpha_{i}^{\vee}(-a_{i}^{-1})\prod_{i=1}^{n-1}x_{-\alpha_i}\Big((-\alpha_i)(\alpha_{n}^{\vee}(-a_{n}^{-1})a_i'\Big)x_{-\alpha_n}(a_{n}^{-1}).
\end{multline*}
This prove \ref{equation: coordinate change distinct simple reflections}.

Under the identification of (\ref{identification.gb}), for any $(u_1,a_1, ..., a_n, t,u)\in wU^{+}B$, 
\begin{equation*}
    (u_1,a_1, ..., a_n, t,u)=u_1 \prod_{i=1}^{n}s_{\alpha_i}x_{\alpha_i}(a_i)tu= u^{+} \prod_{i=1}^{n}\alpha_{i}^{\vee}(-a_{i}^{-1})t u^{-}
\end{equation*}
for some $u^{+}\in U^+$ and some $u^{-}\in U$. Hence, the definition of $f_{\lambda}$ on $U^+ B$ implies that on the subvariety $U^{+,\alpha_1, s_{\alpha_1}(\alpha_2), s_{\alpha_1}s_{\alpha_2}(\alpha_3), ..., s_{\alpha_1}\cdots s_{\alpha_{n-1}}(\alpha_n)} \times (k\setminus\{0\})^{n} \times T \times U \subset wU^{+}B$ the function $f_{\lambda}$ is defined by
\begin{equation*}
    f\Big(u_1 \prod_{i=1}^{n}s_{\alpha_i}x_{\alpha_i}(a_i)tu\Big) = \lambda\Big(\prod_{i=1}^{n}\alpha_{i}^{\vee}(-a_{i}^{-1})t \Big)^{-1} = \prod_{i=1}^{n} (-a_i)^{\langle \lambda , \alpha_{i}^{\vee} \rangle} \lambda(t)^{-1}.
\end{equation*}
Since $\lambda$ is dominant, $f_{\lambda}$ can be extended to $wU^{+}B$ and for $g \in B^{+}wB$, i.e. $(a_1, ..., a_n) = (0,...,0)$, the order of vanishing of $f_{\lambda}$ is given by $\sum_{i=1}^{n}\langle \lambda , \alpha_{i}^{\vee} \rangle$. 
\end{proof}
\subsection{Repeated root reflections}
The following addresses the issue when a simple reflection occurs twice. It is used as the base step in the induction process of \ref{odd.case.comp} and \ref{even.case.comp} below.
\begin{proposition}\label{case.n=1}
Let $\lambda\in X^{*}_+(T)$ and let $f_{\lambda}\in H^{0}(\lambda)_\lambda$. 
If $w = s_{\alpha} s_{\beta} s_{\alpha}$ is a product of simple root reflections with $\alpha\neq \beta$, then the order of vanishing of $f_{\lambda}$ at every point of $B^{+}wB$ is 
    \begin{equation}
\label{eq-alpha-beta-alpha}
        \ord(f_\lambda) = \langle \lambda , \beta^{\vee} \rangle + \langle \lambda , \alpha^{\vee} \rangle \min\{2 , -\langle \alpha , \beta^{\vee} \rangle \}.
    \end{equation}
\end{proposition}
\begin{proof}
The set $R^{+,s_\beta s_\alpha}$ is closed. 
Similar to the previous case, 
\begin{equation*}
\begin{aligned}
    s_\alpha s_\beta s_\alpha U^{+} = s_\alpha s_\beta U^{+,\alpha}s_\alpha U_\alpha &= s_\alpha U^{+, \beta , s_\beta(\alpha)}s_\beta U_\beta s_\alpha U_\alpha \\
    &= U^{+, \alpha, s_\alpha(\beta), s_\alpha(s_\beta(\alpha))}s_\alpha U_\alpha s_\beta U_\beta s_\alpha U_\alpha.
\end{aligned}
\end{equation*}
This yields the isomorphism of varieties
\begin{equation}\label{want.to.rewrite}
\begin{aligned}
    U^{+, \alpha, s_\alpha(\beta), s_\alpha(s_\beta(\alpha))}(k) \times k^3 \times T(k) \times U(k) &\to s_\alpha s_\beta s_\alpha U^{+}(k)B(k) \\
    (u_1,a_1,a_2,a_3,t,u) &\mapsto u_1s_\alpha x_\alpha(a_1) s_\beta x_\beta(a_2) s_\alpha x_\alpha(a_3) tu.
\end{aligned}
\end{equation}
Again we will rewrite the image of this map to obtain an expression in $U^{+}TU$, but because $s_\alpha x_\alpha$ occurs twice, the procedure becomes more complicated since \ref{com.relations} does not hold for $\alpha$ and $-\alpha$. 
Note that in the final expression, the value that the root group morphisms $x_{\alpha_i}$ take does not affect the order of vanishing of $f_\lambda$. To ease the exposition, we will write $x_\alpha$ instead of $x_{\alpha}(-)$ as soon as the value it takes does not affect the order of vanishing of $f_\lambda$. 

With these simplifications, we can rewrite~\eqref{want.to.rewrite} as follows:

\resizebox{0.9\textwidth}{!}{\begin{minipage}{\textwidth}
\begin{equation*}
\begin{aligned}
&\mathrel{\phantom{=}} u_1s_\alpha x_\alpha(a_1) s_\beta x_\beta(a_2) s_\alpha x_\alpha(a_3) tu \\
&= u_1 s_\alpha x_\alpha(a_1) x_\beta(-a_{2}^{-1}) \beta^{\vee}(-a_{2}^{-1})x_{-\beta} x_\alpha(-a_{3}^{-1})\alpha^{\vee}(-a_{3}^{-1})x_{-\alpha}tu \\
&\hspace{6cm} \text{(by (\ref{sx}) for $s_{\beta}x_{\beta}s_{\alpha}x_{\alpha}$)} \\
&=u_1 s_\alpha x_\alpha(a_1) x_\beta(-a_{2}^{-1}) \beta^{\vee}(-a_{2}^{-1}) x_\alpha(-a_{3}^{-1})x_{-\beta}\alpha^{\vee}(-a_{3}^{-1})x_{-\alpha}tu \\
&\hspace{6cm} \text{(because $x_{-\beta}$ and $x_{\alpha}$ commute)} \\
&=u_1 s_\alpha x_\alpha(a_1) x_\beta(-a_{2}^{-1}) x_\alpha\Big((-a_{2})^{-\langle \alpha, \beta^{\vee} \rangle}(-a_{3}^{-1})\Big) \beta^{\vee}(-a_{2}^{-1})\alpha^{\vee}(-a_{3}^{-1})x_{-\beta}x_{-\alpha}tu \\
&\hspace{6cm} \text{(by (\ref{tx}) for $\beta^{\vee}x_{\alpha}$ and $x_{-\beta}\alpha^{\vee}$)} \\
&= u_1 s_\alpha x_{I\alpha + J\beta} x_\beta x_\alpha(a_1)  x_\alpha\Big((-a_{2})^{-\langle \alpha, \beta^{\vee} \rangle}(-a_{3}^{-1})\Big) \beta^{\vee}(-a_{2}^{-1})\alpha^{\vee}(-a_{3}^{-1})x_{-\beta}x_{-\alpha}tu \\
&\hspace{6cm} \text{(by (\ref{com.relations}) for $x_{\alpha}x_{\beta}$)} \\
&= u_1 x_{s_\alpha(I\alpha + J\beta)} x_{s_\alpha(\beta)} s_\alpha x_\alpha(a_1)  x_\alpha\Big((-a_{2})^{-\langle \alpha, \beta^{\vee} \rangle}(-a_{3}^{-1})\Big) \beta^{\vee}(-a_{2}^{-1})\alpha^{\vee}(-a_{3}^{-1})x_{-\beta}x_{-\alpha}tu \\
&\hspace{6cm} \text{(by (\ref{wx}) for $s_{\alpha}x_{I\alpha+J\beta}$ and $s_{\alpha}x_{\beta}$)} \\
&= u_1 x_{s_\alpha(I\alpha + J\beta)} x_{s_\alpha(\beta)} s_\alpha  x_\alpha(d) \beta^{\vee}(-a_{2}^{-1})\alpha^{\vee}(-a_{3}^{-1})x_{-\beta}x_{-\alpha}tu \\
&\hspace{6cm} \text{(
 with $d\coloneqq a_1+(-a_{2})^{-\langle \alpha, \beta^{\vee} \rangle}(-a_{3}^{-1})$)} \\
&=u_1 x_{s_\alpha(I\alpha + J\beta)} x_{s_\alpha(\beta)} x_\alpha \alpha^{\vee}(-d^{-1})x_{-\alpha} \beta^{\vee}(-a_{2}^{-1})\alpha^{\vee}(-a_{3}^{-1})x_{-\beta}x_{-\alpha}tu \\
&\hspace{6cm} \text{(by (\ref{sx}) for $s_{\alpha}x_{\alpha}$)} \\
&=u_1 x_{s_\alpha(I\alpha + J\beta)} x_{s_\alpha(\beta)} x_\alpha \alpha^{\vee}(-d^{-1}) \beta^{\vee}(-a_{2}^{-1})\alpha^{\vee}(-a_{3}^{-1})x_{-\alpha}x_{-\beta}x_{-\alpha}tu \\
&\hspace{6cm} \text{(by (\ref{tx}) for $x_{-\alpha}\beta^{\vee}$ and $x_{-\alpha}\alpha^{\vee}$)} \\
&=u_1 x_{s_\alpha(I\alpha + J\beta)} x_{s_\alpha(\beta)} x_\alpha \alpha^{\vee}(-d^{-1}) \beta^{\vee}(-a_{2}^{-1})\alpha^{\vee}(-a_{3}^{-1})tx_{-\alpha}x_{-\beta}x_{-\alpha}u \\
&\hspace{6cm} \text{(by (\ref{tx}) for $x_{-\alpha}x_{-\beta}x_{-\alpha}t$)}.
\end{aligned}
\end{equation*}
\end{minipage}}

For all $a\in k$, the element $x_{s_{\alpha}(I\alpha + J\beta)}(a) \in U^{+}(k)$ since $i,j>0$ and $s_{\alpha}$ permutes $R^{+}\setminus \{\alpha\}$. Similarly, $x_{s_\alpha(\beta)}(a) \in U^{+}(k)$. Hence, the expression above lies in $U^{+}(k)T(k)U(k)$.

Let $X\subset \AA^3$ be the subscheme defined by $d=0$. On $$U^{+, \alpha, s_\alpha(\beta), s_\alpha(s_\beta(\alpha))}(k) \times ((k\setminus \{0\})^3\setminus X(k)) \times T(k) \times U(k)$$ the function $f_\lambda$ can be written as 
\begin{equation*}
\begin{aligned}
    &\mathrel{\phantom{=}}f_\lambda\Big(u_1 s_\alpha x_\alpha(a_1) s_\beta x_\beta(a_2) s_\alpha x_\alpha(a_3) tu\Big) \\
    &= \lambda(t)^{-1} (-d)^{\langle \lambda , \alpha^{\vee} \rangle} (-a_{2})^{\langle \lambda , \beta^{\vee} \rangle} (-a_{3})^{\langle \lambda , \alpha^{\vee} \rangle } \\
    &=(-1)^{2\langle \lambda , \alpha^{\vee} \rangle+ \langle \lambda, \beta^{\vee} \rangle}\lambda(t)^{-1} \Big(a_1 a_3 + (-1)^{1-\langle \alpha, \beta^{\vee}\rangle}a_{2}^{-\langle \alpha, \beta^{\vee} \rangle}\Big)^{\langle \lambda, \alpha^{\vee} \rangle} a_{2}^{\langle \lambda , \beta^{\vee} \rangle}.
\end{aligned}
\end{equation*}
Since $\alpha$ and $\beta$ are distinct simple roots, $\langle \alpha, \beta^{\vee} \rangle \leq 0$. Since $\lambda$ is dominant, the function $f_\lambda$ extends to $$U^{+, \alpha, s_\alpha(\beta), s_\alpha(s_\beta(\alpha))}(k) \times k^3 \times T(k) \times U(k) = s_\alpha s_\beta s_\alpha U^{+}(k)B(k).$$ 
Its order of vanishing at $(a_1,a_2,a_3)=(0,0,0)$ is given by~\eqref{eq-alpha-beta-alpha}. 
\end{proof}
%
%
%
%
\subsection{Computations for type \texorpdfstring{$\Bsf_m$}{Bn}} Let $(G,\mu)$ as in \ref{th-main}\ref{item-th-main-odd-orthogonal-Bn}. 
This section describes $f_\lambda\in H^0(\lambda)_{\lambda}$ on $wU^+B$ for all $w \in \iw$. 
\subsubsection{Notation}
\label{Bn-notation}
Let $\alpha_1, ..., \alpha_n, \gamma$ be distinct simple roots. For fixed $a_{i,1},a_{i,2},c \in k$, $i=1,...,n$, define $(E_j)_{j=1}^{n}$  recursively by
\begin{equation*}
\begin{aligned}
    E_1 &= a_{1,1}a_{1,2} + (-1)^{1-\langle \alpha_1, \gamma^{\vee}\rangle}  c^{-\langle \alpha_1, \gamma^{\vee} \rangle}, \\ E_{j+1} &= a_{j+1,1}a_{j+1,2} + (-1)^{1-\langle \alpha_{j+1}, \gamma^{\vee}\rangle}  c^{-\langle \alpha_{j+1},\gamma^{\vee}\rangle} \prod_{i=1}^{j}E_i^{-\langle \alpha_{j+1},\alpha_{i}^{\vee} \rangle},
\end{aligned}
\end{equation*}
and let
\begin{equation*}
    t_n \coloneqq \alpha_n^{\vee}(E_n^{-1})\cdots \alpha_1^{\vee}(E_1^{-1}) \gamma^{\vee}(- c^{-1}).
\end{equation*}
\begin{remark}
The dependence of $E_1,..., E_n$ and $t_n$ on the simple roots $\alpha_1, ..., \alpha_n, \gamma$ is omitted to simplify notation.
\end{remark}
View each $E_i$ as a function in $a_{1,1},a_{1,2}, ..., a_{n,1},a_{n,2},c$. Letting $\ord(E_i)$ denote the order at $0$, 
\begin{equation}\label{order.of.vanishing.Ei}
    \ord(E_1) =  \min\{ 2, -\langle \alpha_1, \gamma^{\vee} \rangle \}, \,\,\, 
    \ord(E_i) = \min\Big\{ 2, - \langle \alpha_i, \gamma^{\vee} \rangle - \sum_{j=1}^{i-1}\langle \alpha_i, \alpha_{j}^{\vee} \rangle \ord(E_j) \Big\}.
\end{equation}
A quick computation shows that
\begin{equation}\label{e.n+1.alpha.n+1}
a_{n+1,1}a_{n+1,2} - \alpha_{n+1}(  t_n) = E_{n+1}.
\end{equation}
\begin{theorem}\label{main.theorem.highest.weight.vector.typeBn}
Let $\lambda\in X^{*}_+(T)$ and $f_{\lambda}\in H^{0}(\lambda)_\lambda$. Let $\beta_1, ..., \beta_l, \alpha_1, ..., \alpha_n, \gamma$ be distinct simple roots. Assume that $w=s_{\beta_1}\cdots s_{\beta_l}s_{\alpha_n}\cdots s_{\alpha_1}s_\gamma s_{\alpha_1}\cdots s_{\alpha_n}$ satisfies~\ref{cond-closed}. The order of vanishing of $f_{\lambda}$ on every point of $B^{+}wB$ is 
\begin{equation*}
        \ord(f_\lambda) = \sum_{j=1}^{l}\langle \lambda, \beta_j^{\vee}\rangle + \langle \lambda, \gamma^{\vee} \rangle + \sum_{i=1}^{n}\langle \lambda, \alpha_{i}^{\vee} \rangle \ord(E_i).
\end{equation*}
\end{theorem}
The proof of \ref{main.theorem.highest.weight.vector.typeBn} uses the following two lemmas.
\begin{lemma}\label{odd.case.comp}
There exists a finite set $\mathcal{B}(n)$ of positive linear combinations of $\alpha_1, ..., \alpha_n,\gamma$, and a subset $\mathcal{B}'(n)\subset \{\alpha_1, ..., \alpha_n, \gamma\}$ such that
\begin{equation}\label{odd.case.comp.main.expression}
    \prod_{i=1}^{n}s_{\alpha_{n+1-i}}x_{\alpha_{n+1-i}}(a_{n+1-i,1})s_\gamma x_\gamma(c) \prod_{i=1}^{n} s_{\alpha_i}x_{\alpha_i}(a_{i,2})=\underbrace{\Big(\prod_{\beta\in \mathcal{B}} x_\beta\Big)}_{\eqqcolon u_{n}^+}t_n \underbrace{\Big(\prod_{\beta'\in \mathcal{B}'} x_{-\beta'}\Big)}_{\eqqcolon u_n}
\end{equation}
\end{lemma}
\begin{proof}
By induction on $n$. The case $n=1$ is \ref{case.n=1}. Suppose that the statement holds for some $n \geq 1$. Let $u_{n+1}\coloneqq u_nx_{-\alpha_{n+1}}$ and $u_{n+1}^{+}\coloneqq u^+ x_{\alpha_{n+1}}$, where $u^+$ is defined via 
\begin{equation}\label{u+.def}
    u^+ s_{\alpha_{n+1}}x_{\alpha_{n+1}}=s_{\alpha_{n+1}}x_{\alpha_{n+1}}u_{n}^+,
\end{equation}
by referring to \ref{com.relations} and \ref{wx}. More precisely, by the induction hypothesis on $u_n^+$ we can write $u_n^+=\prod_{\tau} x_{\tau}$ for positive roots $\tau$ such that $\tau\notin \Delta\setminus \{ \alpha_1, ..., \alpha_n, \gamma \}$. In particular, $\tau\neq \alpha_{n+1}$. By \ref{com.relations},  
\begin{equation*}
    x_{\alpha_{n+1}} u_n^+ 
    =\prod_\tau\prod_{i,j>0} x_{i\tau+j\alpha_{n+1}} x_{\tau} x_{\alpha_{n+1}}
\end{equation*}
for some $i=i(\tau),j=j(\tau)>0$. 
Using \ref{wx} we get that
\begin{equation*}
\begin{aligned}
    s_{\alpha_{n+1}}x_{\alpha_{n+1}} u_{n}^+
    &= \underbrace{\prod_\tau \prod_{i,j>0} x_{s_{\alpha_{n+1}}(i\tau+j\alpha_{n+1})} x_{s_{\alpha_{n+1}}(\tau)} }_{u^+} s_{\alpha_{n+1}} x_{\alpha_{n+1}}.
\end{aligned}
\end{equation*}
By assumption on $u_{n}^+$ each $\tau$ is a linear combination of the simple roots $\alpha_1, ..., \alpha_{n}, \gamma$. Hence each $i\tau + j \alpha_{n+1}$ is a linear combination of $\alpha_1, ..., \alpha_{n+1}, \gamma$. Since $\Delta$ is a linearly independent set and since $i,j>0$, one has $i\tau+j\alpha_{n+1}\neq \alpha_{n+1}$. Since $s_{\alpha_{n+1}}$ permutes $R^+\setminus\{ \alpha_{n+1} \}$ the roots $s_{\alpha_{n+1}}(i\tau + j \alpha_{n+1})$ and $s_{\alpha_{n+1}}(\tau)$ are positive for all $\tau$ and all $i,j$ and each is a linear combination of $\alpha_1, ..., \alpha_{n+1}, \gamma$. Let $\mathcal{B}(n+1)$ (resp. $\mathcal{B}'(n+1)$) the set of roots occurring in the expression of $u_{n+1}^{+}$ (resp. $u_{n+1}$) as a product of root groups.

We prove \eqref{odd.case.comp.main.expression} similarly to \ref{case.n=1}. We have 

 \resizebox{0.9\textwidth}{!}{\begin{minipage}{\textwidth}
\begin{equation*}
\begin{aligned}
&\prod_{i=1}^{n+1}s_{\alpha_{n+2-i}}x_{\alpha_{n+2-i}}(a_{n+2-i,1})s_\gamma x_\gamma(c) \prod_{i=1}^{n+1} s_{\alpha_i}x_{\alpha_i}(a_{i,2}) \\
 &= u^{+}s_{\alpha_{n+1}}x_{\alpha_{n+1}}(a_{n+1,1})t_n  u_n s_{\alpha_{n+1}}x_{\alpha_{n+1}}(a_{n+1,2})\\
 & \hspace{6cm} \text{(by induction and definition of $u^+$)} \\
&=u^{+}s_{\alpha_{n+1}}x_{\alpha_{n+1}}(a_{n+1,1}) x_{\alpha_{n+1}}\Big((-a_{n+1,2}^{-1})\alpha_{n+1}(t_n )\Big)t_n  \alpha_{n+1}^{\vee}(-a_{n+1,2}^{-1}) u_n x_{-\alpha_{n+1}} \\
& \hspace{6cm} \text{(by \ref{sx} for $\alpha_{n+1}$ and then \ref{tx} for $t_n $)} \\
&=u^{+}x_{\alpha_{n+1}} \alpha_{n+1}^{\vee}\Big(((-a_{n+1,2}^{-1})E_{n+1})^{-1}\Big)x_{-\alpha_{n+1}}\alpha_{n+1}^{\vee}(-a_{n+1,2}^{-1})t_n   u_n x_{-\alpha_{n+1}} \\
& \hspace{6cm} \text{(by \Cref{e.n+1.alpha.n+1} and then by \ref{sx} for $\alpha_{n+1}$)} \\
&=u_{n+1}^{+} t_{n+1} u_{n+1},
\end{aligned}
\end{equation*}
\end{minipage}}
\vspace{1\baselineskip}\\
where the last equality follows from the definition of $u_{n+1}^+$ and $u_{n+1}$.
\end{proof}
\begin{lemma}\label{odd.case.comp.cor1}
Let $\beta_1, .., \beta_l, \alpha_1, ..., \alpha_n, \gamma$ be distinct simple roots, and let $u_{n}^{+}$, $t_n$ and $u_n$ be as above. Then
\begin{equation*}
\prod_{j=1}^{l}s_{\beta_j}x_{\beta_j} u_{n}^{+}t_n  u = v^{+} \prod_{j=1}^{l}\beta_{j}^{\vee} t_n  v,
\end{equation*}
where $v^{+}=\prod_\eta x_{\eta}$ with each $\eta$ a positive root which is a linear combination of the $\alpha_1, ..., \alpha_n, \beta_1, ..., \beta_l$ and $v =\prod_{\eta'} x_{-\eta'}$ with each $\eta'$ a simple root among $\alpha_1, ..., \alpha_n, \beta_1, ..., \beta_l$.
\end{lemma}
\begin{proof}
This follows by combining the proof of \ref{distinct.simple} with \ref{odd.case.comp}.
\end{proof}
%
%
\begin{proof}[Proof of \Cref{main.theorem.highest.weight.vector.typeBn}]
By the previous lemmas and~\ref{cond-closed}, we may rewrite any element in $wU^{+}U$ in the form $u^{+} \prod_{j=1}^{l}\beta_{j}^{\vee} t_n  u$. If $x_{\beta_i}$ takes the argument $b_i\in k$, for all $i=1,...,l$, then  
\begin{equation*}
    f_{\lambda}(u^{+} \prod_{j=1}^{l}\beta_{j}^{\vee} t_n  u) = \prod_{j=1}^{l}(-b_j)^{\langle \lambda, \beta_j^{\vee}\rangle } \lambda(t_n )^{-1}.
\end{equation*} 
The result follows from the definition~\ref{Bn-notation} of $t_n$.
\end{proof}
\begin{remark}
The preceding two lemmas are only used in the proof of \ref{main.theorem.highest.weight.vector.typeBn}. However, to extend our results on the order of vanishing to more general Weyl group elements by induction on the length,  it seems easier to first work with the expressions in \ref{odd.case.comp.cor1} and then conclude statements on the order vanishing, rather than working directly with \ref{main.theorem.highest.weight.vector.typeBn}.
\end{remark}
\subsubsection{Computation for the Hasse invariant in type \texorpdfstring{$\Bsf_m$}{Bm}}\label{computations.for.hasse.invariant.typeBn}
We apply \ref{main.theorem.highest.weight.vector.typeBn} when $\type(\Delta)=\Bsf_m$. Consider a root datum of type $\Bsf_m$ with simple roots $e_1-e_2,..., e_{m-1}-e_{m},e_m$. Let 
 $\beta_1 = e_1 - e_2$, ..., $\beta_l = e_l - e_{l+1}$ and $\alpha_n = e_{l+1} - e_{l+2}$, ..., $\alpha_1 = e_{m-1}-e_m$ and $\gamma = e_m$. Plugging into \eqref{order.of.vanishing.Ei} gives $\ord(E_i)=2$ for all $i$.

Let $\lambda = e_1$ be the fundamental weight associated to $e_1-e_2$. We compute the order of vanishing of $f_{\lambda}$ at every point of $B^{+}wB$ for $w \in W$ as in \ref{main.theorem.highest.weight.vector.typeBn}. 
Consider two cases:
\begin{enumerate}
    \item 
    \label{item: vanishing deepest stratum type Bn}
If $\alpha_n = e_1 - e_2$, i.e. if ``there are no $\beta_j$'s'', then \ref{main.theorem.highest.weight.vector.typeBn} implies that
\begin{equation*}
    \ord(f_\lambda) = \langle \lambda, \gamma^\vee \rangle + \sum_{i=1}^{m-1}\langle \lambda , \alpha_i^\vee \rangle \ord(E_i) = 2.
\end{equation*}
\item 
\label{item: vanishing middle strata type Dn}
If $\alpha_n = e_l - e_{l+1}$ for $l>1$, then 
\begin{equation*}
    \ord(f_\lambda) = \sum_j \langle \lambda , \beta_j^\vee \rangle + \langle \lambda, \gamma^\vee \rangle + \sum_i \langle \lambda, \alpha_i^\vee \rangle \ord(E_i) = 1 + 0 + 0\cdot 2 = 1.
\end{equation*}

\end{enumerate}


%
%
%
\subsection{Computations for type \texorpdfstring{$\Dsf_m$}{Dn}}\label{section: computations for type Dn}
We now complement the previous results with the analogous ones to cover the case of even orthogonal groups. Since the proofs are identical, we simply refer to the proofs of the corresponding statement in the odd case.
\subsubsection{Notation}
Let $\alpha_1, ..., \alpha_n, \beta, \gamma$ be distinct simple roots. For fixed $a_{i,1},a_{1,2},b,c \in k$, $i=1,...,n$, let $(F_j)_{j=1}^{n}$ be defined recursively by
\begin{equation}
\begin{aligned}
    F_1 &= a_{1,1}a_{1,2} + (-1)^{1 - \langle \alpha_1, \beta^\vee \rangle - \langle \alpha_1, \gamma \rangle} b^{-\langle \alpha_1, \beta^\vee \rangle} c^{-\langle \alpha_1, \gamma \rangle} \\
    F_{j+1} &= a_{j+1,1}a_{j+1,2} + (-1)^{1-\langle \alpha_1, \beta^\vee \rangle-\langle \alpha_{j+1},\gamma^{\vee}\rangle} b^{-\langle \alpha_1, \beta^\vee \rangle} c^{-\langle \alpha_{j+1},\gamma^{\vee}\rangle} \prod_{i=1}^{j}F_i^{-\langle \alpha_{j+1},\alpha_{i}^{\vee}\rangle},
\end{aligned}
\end{equation}
and let
\begin{equation*}
    t_n \coloneqq \alpha_n^{\vee}( F_n^{-1})\cdots \alpha_1^{\vee}( F_1^{-1}) \beta^\vee(- b^{-1}) \gamma^{\vee}(- c^{-1}).
\end{equation*}
We have that
\begin{equation}\label{order.of.vanishing.Fi}
    \ord(F_1)=\min\{ 2 , -\langle \alpha_1 , \beta^{\vee} \rangle - \langle \alpha_1 , \gamma^{\vee} \rangle \}, \,\,\, \ord(F_i) =  \min\Big\{ 2, - \langle \alpha_i , \beta^{\vee}\rangle - \langle \alpha_i, \gamma^{\vee} \rangle - \sum_{j=1}^{i-1}\langle \alpha_i, \alpha_{j}^{\vee} \rangle \ord(F_j) \Big\}.
\end{equation}
\begin{theorem}\label{main.theorem.highest.weight.vector.typeDn}
Let $\lambda\in X^{*}_+(T)$ and $f_{\lambda}\in H^{0}(\lambda)_\lambda$. Let $\eta_1, ..., \eta_l, \alpha_1, ..., \alpha_n, \beta, \gamma$ be distinct simple roots. Assume that $$w:=s_{\eta_1}\cdots s_{\eta_l}s_{\alpha_n}\cdots s_{\alpha_1}s_\beta s_\gamma s_{\alpha_1}\cdots s_{\alpha_n}$$ satisfies~\ref{cond-closed}. The order of vanishing of $f_{\lambda}$ on every point of $B^{+}wB$ is 
    \begin{equation*}
        \ord(f_\lambda) = \sum_{j=1}^{l}\langle \lambda, \eta_j^{\vee}\rangle + \langle \lambda, \beta^{\vee} \rangle + \langle \lambda, \gamma^{\vee} \rangle + \sum_{i=1}^{n}\langle \lambda, \alpha_{i}^{\vee} \rangle \ord(F_i).
    \end{equation*}
\end{theorem}
\begin{proof}
Similar to \ref{main.theorem.highest.weight.vector.typeBn}, the theorem follows from~\ref{even.case.comp} and~\ref{even.case.comp.cor1}.
\end{proof}
%
%
%
\begin{lemma}\label{even.case.comp}
There exists a finite set $\mathcal{N}$ of positive linear combinations of $\alpha_1, ..., \alpha_n,\beta, \gamma$, and a subset $\mathcal{N}'\subset \{\alpha_1, ..., \alpha_n, \beta, \gamma\}$ such that
\begin{equation}\label{f.lambda.expression.even}
\prod_{i=1}^{n}s_{\alpha_{n+1-i}}x_{\alpha_{n+1-i}}(a_{n+1-i,1})s_\beta x_\beta(b) s_\gamma x_\gamma(c) \prod_{i=1}^{n} s_{\alpha_i}x_{\alpha_i}(a_{i,2})=\underbrace{\Big(\prod_{\eta\in\mathcal{N}} x_\eta\Big)}_{u_{n}^+} t_n \underbrace{\Big(\prod_{\eta'\in\mathcal{N}'}x_{-\eta'}\Big)}_{u_n}.
\end{equation}
\end{lemma}
\begin{proof}
Same as the proof of \ref{odd.case.comp}.
\end{proof}
\begin{lemma}\label{even.case.comp.cor1}
Let $\eta_1, ..., \eta_l, \alpha_1, ..., \alpha_n, \beta, \gamma$ be distinct simple roots. Let $u_{n}^{+}$, $t_n$ and $u_n$ be as in \ref{even.case.comp}. Then
\begin{equation*}
\prod_{j=1}^{l}s_{\eta_j}x_{\eta_j} u_{n}^{+}t_n u_n 
= u^{+} \prod_{j=1}^{l}\eta_{j}^{\vee} t_n u,
\end{equation*}
where $u^{+}=\prod_\zeta x_{\zeta}$ with each $\zeta$ a positive root which is a linear combination of the $\alpha_1, ..., \alpha_n, \eta_1, ..., \eta_l$ and $u =\prod_{\zeta'} x_{-\zeta'}$ with each $\zeta'$ a simple root among $\alpha_1, ..., \alpha_n, \eta_1, ..., \eta_l$.
\end{lemma}
\begin{proof}
Same as the proof of \ref{odd.case.comp.cor1}.
\end{proof}
%
%
\subsubsection{Computation for the Hasse invariant in type \texorpdfstring{$\Dsf_m$}{Dn}}\label{computations.for.hasse.invariant.typeDn}
We make the computations analogous to \ref{computations.for.hasse.invariant.typeBn}. Consider a root datum of type $\Dsf_m$ with simple roots $e_1-e_2, ..., e_{m-1}-e_{m},e_{m-1}+e_m$. Let $\beta = e_{m-1}-e_m$ and let $\gamma = e_{m-1}+e_m$. For $i\geq m+1$, let $j=j(i)=\Big|\{m-2, ..., 2m-1-i\}\Big|$. Let $\alpha_1 = e_{m-2}-e_{m-1}, ..., \alpha_j=e_{2m-1-i}-e_{2m-i}$ and let $\eta_1=e_1-e_2,..., \eta_{m-j}=e_{m-j}-e_{m-j+1}$. Let $\lambda=e_1$. We compute $\ord(f_{\lambda})$ at every point of $B^+wB$ for every $w \in W$ as in  \ref{main.theorem.highest.weight.vector.typeDn}.

By \eqref{order.of.vanishing.Fi}, $\ord(F_i)=2$ for all $i$. Consider two cases:
\begin{enumerate}
    \item 
If $j=m-2$, then $\alpha_j=e_1-e_2$, so ``there are no $\eta_j$'s'' in \ref{main.theorem.highest.weight.vector.typeDn}.  Hence
\[
\ord(f_{\lambda}) = \langle \lambda , \beta^{\vee} \rangle + \langle \lambda , \gamma^{\vee} \rangle + \sum_{i=1}^{m-2}\langle \lambda , \alpha_i^\vee \rangle \ord(F_i) = 2.
\]
\item
If $j<m-2$, then \ref{main.theorem.highest.weight.vector.typeDn} gives 
\[
\ord(f_{\lambda}) = \sum_{j=1}^{l}\langle \lambda, \eta_j^{\vee}\rangle + \langle \lambda , \beta^{\vee} \rangle + \langle \lambda , \gamma^{\vee} \rangle + \sum_{i=1}^{j}\langle \lambda , \alpha_i^\vee \rangle \ord(F_i) = 1.
\]
\end{enumerate}
\subsection{Functoriality of Ogus' principle}\label{section: functoriality of ogus principle, lemma section}
Let $\varphi\colon G\to G'$ be as in \S\ref{section: functoriality of ogus principle, introduction}. 
Let $T'\subset B'$ be a Borel pair in $G'$. Set $T\coloneqq \varphi^{-1}(T')$ and $B\coloneqq \varphi^{-1}(B')$. 
For any $x\in \gzip^{\mu}(k)$ and $x'\in \sbt(k)$, say that $x$ and $x'$ are \textit{flag-related} if there is a $k$-point $x''\in \gzipf^{\mu}(k)$ such that $x=\piflagz(x'')$ and $x'=\piflags(x'')$. If $x$ is flag-related to $x'$, then $\varphi(x)$ is flag-related to $\varphi(x')$. 
\begin{lemma}\label{lemma: functoriality for good w}
Suppose that each $k$-point $x$ in $\gzip^{\mu}$ is flag-related to a $k$-point $x'$ in $\sbt_{w}$ 
for $w\in W$ as in \S\ref{section: product of distinct simple root reflections}-\ref{section: computations for type Dn}. For every $\fp$-representation $r' \colon G'\to \GL(V)$, Ogus' principle for $(G,\mu,r'\circ \varphi)$ is equivalent to Ogus' principle for $(G',\varphi \circ \mu,r')$. 
\end{lemma}
\begin{proof}
Let $x,x'$ and $w$ be as in the hypothesis. 
Since $\clp_x(G,\mu,r\circ \varphi)=\clp_{\varphi(x)}(G',\varphi\circ \mu,r)$ it suffices to show that $\ord_{\varphi(x)}(\ha(G',\varphi\circ \mu,r'))=\ord_{x}(\ha(G,\mu,r'\circ \varphi))$. 
Let $\lambda'\in X_{+}^{*}(T')$ such that $\piflagz^*\ha(G',\varphi\circ \mu,r)=\piflags^*(w_0f_{\lambda'})$. Set $\lambda\coloneqq \lambda'\circ \varphi$. Since $\varphi$ has central kernel and $G$ and $G'$ have the same adjoint group, we see (e.g. by Weyl dimension formula) that $H^0(\lambda)$ is the pullback $\varphi$ of the representation $H^0(\lambda')$. Furthermore, the highest weight space $H^0(\lambda)_{\lambda}$ is the pullback of the weight space $H^0(\lambda')_{\lambda'}$. Thus, $\piflagz^*\ha(G,\mu,r\circ \varphi)=\piflags^*(w_0f_{\lambda})$. 
%
The fact that $G$ and $G'$ have the same adjoint group and the assumption on $w$ implies that the order of vanishing of $f_{\lambda}$ on any point in $B^+wB$ is the same as the vanishing order of $f_{\lambda'}$ on any point in $B'^+wB'$, since both are given by a sum of pairings of $\lambda$ (resp. $\lambda'$) with coroots. Since $\varphi$ induces a morphism between diagram (\ref{eq-zipflag-sbt-diagram}) for $G$ and 
$G'$, we find that 
\[
\ord_{\varphi(x)}(\ha(G',\varphi\circ \mu,r))=\ord_{\varphi(x')}(w_0f_{\lambda'})=\ord_{x'}(w_0f_{\lambda})=\ord_{x}(\ha(G,\mu,r\circ \varphi)).
\]
\end{proof}
\section{Ogus Principle in type \texorpdfstring{$\Bsf_m$}{B} and \texorpdfstring{$\Dsf_m$}{D}}
\label{sec-orthogonal-case}
%
\subsection{Group-theoretic notation}\label{section: orthogonal an gspin varieties}
Let $\GG_0$ be a connected, reductive $\QQ$-group such that $\GG_{0,\RR} \cong \SO(n,2)$, the special orthogonal group of signature $(n,2)$. Let $\tilde \GG_0$ be the simply-connected double cover of $\GG_0$. Let $\GG:=(\tilde \GG_0 \times \mg_m)/\langle -1,-1 \rangle$ be the associated spin similitude group. The $\RR$-group $\GG_{\RR} \cong \GSpin(n,2)$. Let $\gx$ (resp. $(\GG_0,\XX_0$)) be a Shimura datum with group $\GG$ (resp. $\GG_0)$. The datum $\gx$ (resp. $(\GG_0, \XX_0$) is of Hodge (resp. abelian) type.

Assume that $\GG_0$ is unramified at $p$. Applying \S\ref{sec-intro-Hodge-type} to $\gx$ and $(\GG_0, \XX_0)$ gives cocharacter data $(G,\mu)$, $(G_0,\mu_0)$, special $k$-fibers $S_K$, $S_{K_0}$ of the associated Shimura varieties and smooth surjective Zip period maps $\zeta_K, \zeta_{K_0}$ respectively. The cocharacter  $\mu_0\colon \mg_m \to G_{0,k}$ is given in coordinates by $(1,0,...,0,-1)$.

\subsection{The Hodge character in the spin and half-spin cases}
For $n=2m+1$ odd (resp. $n=2m$ even), let 
  $\spin\colon G\to \GL(2^{m})$ (resp. $\spin^{\pm}: G \to \GL(2^{m-1})$) be the spin representation (resp. the two half-spin representations) of $G$. 
Identify
\[
X^*(T_{\tilde G}) = \{ (a_1, ..., a_m) \in \frac{1}{2}\mz^m : \text{ $a_i \in \mz$ for all $i$ or $a_i \in \frac{1}{2}\mz$ for all $i$} \}
\]
so that $R=\{\pm(e_i \pm e_j) |1 \leq i <j \leq m\} \cup \{\pm e_i | 1\ \leq i \leq m\} $ for $n=2m+1$ odd and $R=\{\pm(e_i \pm e_j) |1 \leq i <j \leq m\}$ for $n=2m$ even.
Recall the Hodge character $\eta \in X^*(L) $~\eqref{sec-hodge-char}.
\begin{proposition} \
\begin{enumerate}
\item
\label{item-spin-odd}
If $n=2m+1$ is odd, then the Hodge character of $(G,\mu,\spin)$ is $\eta=-2^{m-2}e_1$.
\item
\label{item-spin-even}
If $n=2m$ is even, then the Hodge character of $(G,\mu, \spin^{\pm})$ is $\eta=-2^{m-3}e_1$.    
\end{enumerate}    
\end{proposition}
\begin{proof}
The Hodge character $\eta$ is the sum of the $T$-weights $\chi$ such that $\langle \chi, \mu\rangle=-1$. Since $\spin$ and $\spin^{\pm}$ are minuscule, their weights are the $W$-orbit of the highest weight. The set of $2^m$ weights of $\spin$ is $\{\frac{1}{2}(\pm 1, \ldots, \pm 1)\}$. The half with first coordinate $-1$ are precisely those that satisfy $\langle \chi, \mu\rangle=-1$. Their sum is $\eta=-2^{m-2}e_1$. The same argument applies to the half-spin representations, whose dimension is $2^{m-1}$. Hence $\eta=-2^{m-3}e_1$ in case $n=2m$ is even.         
\end{proof}


\subsection{The Bruhat stratification}\label{section: bruhat stratification orthogonal case}
We give an account of the Bruhat stratification as described in \cite[Section 3]{wedhorn.bruhat} (see also \cite[Section 7]{shen.zhang}). 
Since $G$ is a central extension of $G_0$ by $\mg_m$, it suffices to describe the case of $G_0$.  
%

{\bf Case $n$ odd:} Identify the Weyl group $W$ with the subset of the symmetric group $S_{n+2}$ given by
\begin{equation*}
    W = \{ \sigma \in S_{n+2} : \sigma(i) + \sigma(n+3-i)=n+3, i=1,...,n+2\}.
\end{equation*}
Let $m = \frac{n+1}{2}$ and for each $i=1,...,m-1$ let $s_i$ denote the reflection
\begin{equation}\label{odd.case.si}
    s_i = (i,i+1)(n+2-i, n+3-i)
\end{equation}
and let
\begin{equation*}
    s_m = (m,m+2).
\end{equation*}
The Weyl group $W$ is generated by $s_1, ..., s_m$ and we have that
\begin{equation*}
    I = J = \{ s_2, ..., s_m \}.
\end{equation*}
The length function induces a bijection of totally ordered sets (see \cite[Lemma 3.6]{wedhorn.bruhat})
\begin{equation*}
    {}^{I}W \to \{0,...,n\}, w\mapsto l(w)
\end{equation*}
with inverse
\begin{equation}\label{e.o.stratification.odd.orthogonal}
\begin{aligned}
    \{ 0,..., 2m-1\} &\to {}^{I}W \\
    d&\mapsto 
    \begin{cases} s_1 \cdots s_d & 1\leq d \leq m \\
    s_1\cdots s_m s_{m-1}\cdots s_{2m-d} & m+1\leq d \leq 2m-1 
    \end{cases}
\end{aligned}
\end{equation}
Let $w_{0,I,J}$ be the longest element in ${}^{I}W^{J}$. We have bijections
\begin{equation*}
    \{ 1 , s_1, w_{0,I,J} \} = {}^{I}W^{J} \to \{ 0,1,n\}.
\end{equation*}
{\bf Case $n$ even:} Let $s_i$ be as in (\ref{odd.case.si}) for $i=1,...,m-1$, and let
\begin{equation*}
    s_m \coloneqq (m-1,m+1)(m,m+2).
\end{equation*}
For all $i=1,...,m$, let 
\begin{equation*}
    w_i \coloneqq \begin{cases} s_1\cdots s_i & i \leq m \\
    s_1 \cdots s_m s_{m-2} \cdots s_{2m-1-i} & i \geq m+1
    \end{cases}
\end{equation*}
and let
\begin{equation*}
    w_{m-1}' \coloneqq s_1\cdots s_{m-2}s_m.
\end{equation*}
Then by \cite[Lemma 3.7]{wedhorn.bruhat} we have that
\begin{equation*}
    {}^{I}W = \{ \id \} \cup \{ w_i \}_{i=1}^{n} \cup \{ w_{m-1}'\}
\end{equation*}
with the order given by
\begin{equation*}
    \begin{aligned}
        \id \leq w_1 \leq ... \leq w_{m-2} \leq w_{m-1},w_{m-1}' \leq w_m \leq ... \leq w_n.
    \end{aligned}
\end{equation*}
The indexing set for the Bruhat stratification is given by
\begin{equation*}
    {}^{I}W^{J} = \{ 0 , w_1, w_n \},
\end{equation*}
where $w_n = w_{0,I,J}$ is the longest element in ${}^{I}W^{J}$.
%
%
%
%
%
%
%
\subsection{Order of vanishing of the Hasse invariant}
Keep the setting of \S\ref{section: bruhat stratification orthogonal case}. In particular, $w_0, w_{0,I}$ commute and $z=z^{-1}$.  
By \cite[Definition 9.4]{pink.wedhorn.ziegler.zip.data}, there is an equivalence relation on $W$ given by 
\[
w\sim w' \iff \text{there exists $y\in W_I$ s.t. $yw'z\varphi(y)^{-1}z=w$}.
\]
 For any $w,w'\in W$, by \cite[Theorem 9.17]{pink.wedhorn.ziegler.zip.data} we have that $w$ and $w'$ lie in the same Ekedahl-Oort stratum of $\gzip^{{}^{z}\mu}$ if and only if $zw\sim zw'$. 
\begin{remark}
Throughout this text we have used the frame $(B,T,z)$. To use the results from \cite{pink.wedhorn.ziegler.zip.data} in the following lemma we conjugate by $z$ to use the frame $({}^{z}B,T,z)$.
\end{remark}
\begin{lemma}\label{lemma: from sbt to gzips in orthogonal case}
For all $w,w'\in {}^{I}W$, the elements $w_0wz$ and $w_0w'z$ lie in the same $E$-orbit of $G$  if and only if $w=w'$.
\end{lemma}
\begin{proof}
By passing to the frame $({}^{z}B,T,z)$ it suffices to show that $zw_0w$ and $zw_0w'$ lie in the same Ekedahl-Oort stratum of $\gzip^{{}^{z}\mu}$ if and only if $w=w'$. 
If $n$ is odd, then $w_0=-1$ is central. Hence $zw_0w\sim zw_0w'$ if and only if $zw\sim zw'$. Since $zw,zw'\in {}^{I}W$ this implies that $zw=zw'$. Thus $w=w'$. 

If $n$ is even, then $w_0$ is central if and only if $\frac{n+2}{2}=m$ is also even, in which case the previous argument applies. Assume that $m$ is odd. Then $w_{0,I}$ is central in $W_I$ (since $L$ is of type $\Dsf_{m-1}$). Since $zw_0=w_{0,I}$ we see that $zw_0w$ and $zw_0w'$ lie in the same Ekedahl-Oort stratum if and only if $w_{0,I}w\sim w_{0,I}w'$. Since $w_{0,I}$ is central in $W_I$, this is true if and only if $w=w'$.
\end{proof}
\begin{theorem}\label{theorem: order of vanishing of hasse invariant on stack of G-zips for Bn and Dn Shimura datum}
Let $(G,\mu)$, $(G_0,\mu_0)$ as in  \S\ref{section: orthogonal an gspin varieties}. 
Let $m'=m$ if $n$ is odd and $m'=m-1$ if $n$ is even. For all $x\in \gzip^\mu(k)$ (resp. $x\in \gozip^{\mu_0}(k)$),
\[
\ord_x(\ha(G,\mu,\spin)) = \begin{cases} 0 & x \in \gzip_{w_{0,I,J}}^{\mu,\bruh} \\
                                2^{m'-2} & x \in \gzip_{s_1}^{\mu,\bruh} \\
                                2^{m'-1} & x\in \gzip_{1}^{\mu,\bruh}
                                \end{cases} ,
                                \,\,\,\,\,
\ord_x(\ha(G_0,\mu_0,\std)) = \begin{cases} 0 & x \in \gozip_{w_{0,I,J}}^{\mu_0,\bruh} \\
                                1 & x \in \gozip_{s_1}^{\mu_0,\bruh} \\
                                2 & x\in \gozip_{1}^{\mu_0,\bruh}
                                \end{cases}.
\]
\end{theorem}
\begin{proof}
Let $\eta_0 \coloneqq -e_1$. Since $\eta = 2^{m'-2}\eta_0$ the computations for the orthogonal case yield also the result for the GSpin case. So it suffices to prove the statement for $(G_0,\mu_0,\std)$. 

Since $w_{0,J}$ is a product of simple reflections $s_i$ for $i>1$, we see that $w_{0,J}e_1 = e_1$. Since $\sigma$ fixes $e_1$, 
\begin{equation*}
    D_{w_0}(e_1)  = e_1 - pzw_0 e_1 = e_1 - pw_{0,J} e_1 = (p-1)\eta_0.
\end{equation*}
For all $w\in {}^{I}W$, \ref{computations.for.hasse.invariant.typeBn} and \ref{computations.for.hasse.invariant.typeDn} give the order of vanishing for $f_{e_1}$ on all points in $B^{+}wB$. Since $w_0e_1=-e_1$, \S\ref{general.strategy} implies that for any $w$ in ${}^{I}W$, the order of vanishing of $\ha(G_0,\mu_0,\std)$ on the zip stratum containing $w_0wz$ is gotten from the order of $f_{e_1}$ on any point in $B^{+}wB$. \Cref{lemma: from sbt to gzips in orthogonal case} gives the order of vanishing of $\ha(G_0,\mu_0,\std)$ on all zip strata.

For odd $n$, \ref{computations.for.hasse.invariant.typeBn}\ref{item: vanishing deepest stratum type Bn} gives the order of vanishing of $f_{e_1}$ on the longest element in ${}^{I}W$, hence gives the order of vanishing of $\ha(G_0,\mu_0,\std)$ at the smallest zip stratum. The vanishing order of $\ha(G_0,\mu_0,\std)$ on all the strata lying in-between the open dense stratum and the closed stratum is given by \ref{computations.for.hasse.invariant.typeBn}\ref{item: vanishing middle strata type Dn}. For even $n$ we refer to \ref{computations.for.hasse.invariant.typeDn}. Under the bijection (\ref{e.o.stratification.odd.orthogonal}) the Bruhat stack is indexed by $\{ 0 ,1, n\}$, and the statement follows.
\end{proof}
\subsection{The conjugate line position}
Keep the notation of~\S\ref{section: orthogonal an gspin varieties}. For every $x\in \gozip^{\mu_0}(k)$, let 
\[
\underline{M}_{\mu_0}^{w}(x)=(M_{\mu_0}^{w}(x),\fil^{\bullet}M_{\mu_0}^{w}(x), \fil_{\Conj,\bullet}M_{\mu_0}^{w}(x),(\varphi_\bullet)_{\mu_0}^{w}(x))
\]
be the $F$-zip obtained from the universal $G_0$-zip by pushforward along $r_0\colon G_0\to \GL(n+2)$. 
\begin{theorem}\label{theorem: conjugate line position orthogonal case}
For every $x\in\gozip^{\mu_0}(k)$, the conjugate line position of the corresponding $F$-zip is given by
\begin{equation}
    \clp(\underline{M}_{\mu_0}^{w}(x)) = \begin{cases} 0 & x \in \gozip_{w_{0,I,J}}^{\mu_0,\bruh} \\
                                1 & x \in \gozip_{s_1}^{\mu_0,\bruh} \\
                                2 & x\in \gozip_{1}^{\mu_0,\bruh}
                                \end{cases}.
\end{equation}
\end{theorem}
\begin{proof}
The proof follows by translating \cite[Section 2.6]{moonen.wedhorn} into this particular setting, as follows. 

The cocharacter $\mu_0 \in X_*(G_{0})=(1,0,...,0,-1)$. So the type of $\underline{M}_{\mu_0}^{w}(x)$ is 
\[
\tau \colon \mz\to \mz_{\geq 0},\,\,\, \tau(i) = \begin{cases}
    1 & i = 0,2 \\
    n & i = 1 \\
    0 & \text{otherwise}
\end{cases}
\]
Hence, in the notation of \cite[Section 2.6]{moonen.wedhorn} we have that 
\[
i_1 = 0, i_2 = 1, i_3=2.
\]
Thus
\[
n_1 = 1, n_2 = n, n_3 = 1
\]
and 
\[
m_1 = 1, m_2 = m_1 + n_2 = n+1, m_3 = m_2 + n_3 = n+2.
\]
The element $w_{0,I,J}$, which is of minimal length in $w_0 W_I$, is thus given by
\begin{equation*}
    w_{0,I,J}(i) = i + m_j + m_{j-1} - (n+2) , \,\,\, (n+2)-m_j < i \leq (n+2)-m_{j-1},
\end{equation*}
where $m_0\coloneqq 0$. Hence,
\begin{equation*}
    w_{0,I,J}(i) = \begin{cases}
    n+2 & i = 1 \\
    i & i=2,...,n+1 \\
    1 & x = n+2
    \end{cases}
\end{equation*}
For each tuple $(a_1, ..., a_m)$ with $m\leq n+2$, let $k^{\{a_1, ..., a_m\}}$ denote the subspace of $k^{n+2}$ spanned by the $a_1,...,a_m$-basis vectors (in a fixed basis of $k^{n+2}$). For each $w\in {}^{I}W$, following \cite[Section 1.9]{moonen.wedhorn} we identify $\underline{M}_{\mu_0}^{w}(x)$ with the following $F$-zip:
\begin{enumerate}
    \item $M_{\mu_0}^{w}(x) = k^{n+2}$
    \item $\fil^{\bullet}M_{\mu_0}^{w}(x): k^{n+2} \supset k^{\{ w(1), ..., w(m_2) \}} \supset k^{\{w(1)\}} \supset 0$ 
    \item $\fil_{\Conj,\bullet}M_{\mu_0}^{w}(x): 0 \subset k^{\{1\}} \subset k^{\{1,...,n+1\}} \subset k^{n+2}$
    \item $(\varphi_i)_{\mu_0}^w(x)$ is zero for $i\notin \{ i_1, i_2, i_3\}=\{ 0,1,2\}$ and for $i =i_j\in \{ i_1,i_2,i_3\}$ it is the isomorphism
\begin{equation*}
\begin{aligned}
    (\gr^{i}{\fil_{\bullet}M_{\mu_0}^{w}}(x))^{(p)} &= k^{\{ w(m_{3-j}+1), ..., w(m_{3-j+1})\}} \to \\
    & \,\,\,\,\,  \gr_{i}\fil_{\Conj,\bullet}M_{\mu_0}^{w}(x) = k^{\{ n+2-m_{3-j+1}+1, ..., n+2-m_{3-j}\}}
\end{aligned}
\end{equation*}
    induced by the permutation matrix associated to $w_{0,I,J}^{-1}w^{-1}$. 
\end{enumerate}
Hence, for each $w\in {}^{I}W$, to find the conjugate line position of $\underline{M}_{\mu_0}^{w}(x)$ it suffices to determine whether $1$ lies in $\{w(1)\}$ or in $\{w(1), ..., w(n+1)\}$, or in neither of them. If $w=1$ then $1 \in \{ w(1) \}$ so $a(\underline{M}_{\mu_0}^{1}(x))=2$. If $w=w_{0,I,J}$ is the longest element in ${}^{I}W$, then $\{ w(1), ..., w(m_2)\} = \{ n+2,2,...,n+1\}$ so $a(\underline{M}_{\mu_0}^{w_{0,I,J}}(x))=0$. If $w\neq1,w_{0,I,J}$ with $l(w)=d$, then we may write
\[
w = \begin{cases} s_1 \cdots s_d & 1\leq d \leq m \\
    s_1\cdots s_m s_{m-1}\cdots s_{2m-d} & m+1\leq d < 2m-1 
    \end{cases}
\]
(we only give the description in type $B_n$, but the arguments applies verbatim in type $D_n$). 
Hence, there is some $t_1\in \{2,...,n+1\}$ such that $w(t_1)=1$. Since $t_1\neq 1,n+2$, one has $1\notin \{w(1)\}$. But $1\in \{w(1), ..., w(n+1)\}$. Hence $\clp(\underline{M}_{\mu_0}^{w}(x))=1$.
\end{proof}
\begin{theorem}[\ref{th-main}\ref{item-th-main-odd-orthogonal-Bn},\ref{item-th-main-even-orthogonal-Dn}]\label{corollary: ogus principle for Bn}
The triple $(G_0,\mu_0,\std)$ satisfies Ogus' principle.
\end{theorem}
\begin{proof}
This follows by combining \ref{theorem: order of vanishing of hasse invariant on stack of G-zips for Bn and Dn Shimura datum} and \ref{theorem: conjugate line position orthogonal case}.
\end{proof}
\begin{corollary}[\ref{th-main}\ref{item-th-main-GL(4)}]\label{corollary: main theorem for GL(4)}
The triple $(\GL(4),\mu,\wedge^2)$ of \ref{th-main}\ref{item-th-main-GL(4)} satisfies Ogus' Principle.
\end{corollary}
\begin{proof}
Since $\GL(4)$ and $\SO(6)$ have the same adjoint group, the result follows from \ref{lemma: functoriality for good w}.
\end{proof}
\subsection{Application to orthogonal and unitary Shimura varieties}
Let $\ha_{S_{K_0}}\coloneqq \zeta_{K_0}^*\ha(G_0,\mu_0,\std)$. Let $\ha_{S_{K}}\coloneqq \zeta_{K}^*\ha(G,\mu,\spin)$.  
\begin{corollary} 
Ogus' Principle holds for $S_{K_0}$: for every $x\in S_{K_0}(k)$,
\[
\ord_x(\ha_{S_{K_0}})=\clp_x(G_0,\mu_0,\std).
\]
Furthermore, 
the vanishing order of the Hasse invariant $\ha_{S_K}$ on $S_K$ (resp.  $\ha_{S_{K_0}}$ on $S_{K_0}$) is given for all $x\in S_K(k)$ (resp. $x\in S_{K_0}(k)$) by
\[
\ord_x(\ha_{S_K}) = \begin{cases} 0 & x \in S_{K,w_{0,I,J}}^{\bruh} \\
                                2^{m-2} & x \in S_{K,s_1}^{\bruh} \\
                                2^{m-1} & x\in S_{K,1}^{\bruh}
                                \end{cases} ,
                                \,\,\,\,\,
\ord_x(\ha_{S_{K_0}}) = \begin{cases} 0 & x \in S_{K_0,w_{0,I,J}}^{\bruh} \\
                                1 & x \in S_{K_0,s_1}^{\bruh} \\
                                2 & x\in S_{K_0,1}^{\bruh}
                                \end{cases}.
\]
\end{corollary}
\begin{proof}
Since the map $\zeta_K$ (resp. $\zeta_{K_0}$) is smooth, this follows from \ref{theorem: order of vanishing of hasse invariant on stack of G-zips for Bn and Dn Shimura datum} and \ref{corollary: ogus principle for Bn}.
\end{proof}
\begin{corollary}
Let $S_K$ be the $k$-special fibre of a unitary Shimura variety of signature $(2,2)$, as in \ref{cor-intro-Hodge-type}\ref{item: cor hodge type unitary}. Then $S_K$ satisfies Ogus' Principle. The Hasse invariant exactly assumes the values $0,1,2$ and the vanishing order at $x,x'\in S_K(k)$ are equal if and only if $x$ and $x'$ lie in the same Bruhat stratum.
\end{corollary}
\begin{proof}
This follows from \ref{corollary: main theorem for GL(4)}.
\end{proof}
%
%
%
%
\subsection{Ogus' Principle for K3 surfaces via Zip period maps}
\label{sec-K3}
Let $M^{\circ}_{2d, K}$ denote the moduli space of $K3$-surfaces equipped with a $K_0^p$-level structure, and a polarization of degree $2d$ (see \cite[Section 3]{Madapusi-Tate-K3}). 
By
\cite[Corollary 5.15]{Madapusi-Tate-K3}, the Kuga-Satake construction induces a (non-canonical) open immersion 
\begin{equation}
\iota_{K_0}^{KS}:M^{\circ}_{2d, K} \to S_{K_0}.    
\end{equation}
The composition $\zeta_{K3}:M^{\circ}_{2d,K} \to \gozip^{\mu_0}$ is a smooth and surjective morphism, 
and the following diagram commutes 
\begin{equation*}
    \begin{tikzcd}
         M_{2d,K_0}^{\circ} \arrow[r, "\iota_{K_0}^{KS}"] \arrow[rd, "\zeta_{K3}"'] & S_{K_0} \arrow[d, "\zeta_{K_0}"] \\
         & \gozip^{\mu_0} 
    \end{tikzcd}
\end{equation*}
%
%
Let $\ha_{M_{2d,K}^{\circ}}:=\zeta_{K3}^*\ha(G_0,\mu_0,\std)$ be the Hasse invariant on $M_{2d,K}^{\circ}$. 
\begin{corollary}\label{corollary: hasse vanishing for k3}
Ogus Principle holds for $M_{2d,K}^{\circ}$. For all $x \in M_{2d,K}^{\circ}(k)$, 
\begin{equation*}
    \ord_x(\ha_{M_{2d,K}^{\circ}}) = \clp_x(H_{\dr}^2(X/M_{2d,K}^{\circ}))= \begin{cases} 0 & x \in (M_{2d,K}^{\circ})_{w_{0,I,J}}^{\bruh} \\
                                1 & x \in (M_{2d,K}^{\circ})_{s_1}^{\bruh} \\
                                2 & x\in (M_{2d,K}^{\circ})_{1}^{\bruh}
                                \end{cases}.
\end{equation*}
\end{corollary}
\begin{proof}
Since $\zeta_{K3} = \zeta_{K_0}\circ \iota_{K_0}^{KS}$ this follows from \ref{corollary: ogus principle for Bn} and \ref{theorem: order of vanishing of hasse invariant on stack of G-zips for Bn and Dn Shimura datum}.
\end{proof}
%
%
%
%
%
%
%
%
%
%
%
%
%
%
%
%
%
%
%
%
%
%
%
%
%
\section{Ogus Principle in type \texorpdfstring{$\Asf_n$}{A}}\label{unitary.case}
\subsection{Group-theoretic notation}\label{section: group notation unitary case}
Let $\mathbf{G}$ be a connected, reductive $\mq$-form of $\GL(n)$ such that $\mathbf{G}_\mr \cong \U(a,a')$ for some signature $(a,a')$ with $a+a'\eqqcolon n$. Assume that $\mathbf{G}$ is unramified at $p$ and that the special fiber $G$ is split over $\mf_p$. We have that $G\cong \GL(n)$. Let $T\subset G$ denote the diagonal matrices and $B\supset T$ the lower triangular matrices. Identify $X^{*}(T)$ with $\mz^n$ in the usual way; $e_i : \diag(t_1,...,t_n)\mapsto t_i$. Let $\mu \colon \mg_m\to G_{k}$ be the cocharacter given in coordinates by $(1,...,1,0,...,0)$, where there are $a$ number of $1$'s. Letting $\alpha_i = e_i - e_{i+1}$, we have that $\Delta=\{ \alpha_1, ..., \alpha_{n-1}\}$. The parabolics $P$ and $Q$ determined by $\mu$ are of type $I = \Delta \setminus \{ \alpha_{a}\}$ respectively $J=-w_0I$. For each $i$ let $s_i$ denote the simple reflection corresponding to $s_{\alpha_i}$. Let $\eta \in X^*(T)$ be the character given by $(-1, ..., -1,1, ..., 1)$, with $a$ number of $-1$'s.
\begin{remark}
If instead $\mathbf{G}_\mr\cong \GU(a,a')$, then $G\cong \GL(n)\times \mg_m$. The arguments in \S\ref{unitary.vanishing} applies (cf. \ref{section: functoriality of ogus principle, introduction}). We work with $\GL(n)$ for simplicity.
\end{remark}
\begin{remark}
The fundamental weights are only well-defined up to scaling by characters of the form $(a, ..., a)$ for some $a\in \ZZ$. Since $GL_{n,k}/B=SL_{n,k}/(B\cap SL_{n,k})$ and $(a,...,a)=0 \in X^{*}(T\cap SL_{n,k})$ we see that $H^{0}(\lambda)= H^{0}(\lambda + (a,...,a))$ for all characters $\lambda$. Thus, for any character $\nu_i$ such that $\langle \nu_i , \alpha_{j}^{\vee} \rangle =\delta_{i,j}$ for all simple roots $\alpha_j$ we refer to it as the fundamental weight corresponding to $\alpha_i$. Conversely, we use $\nu_i$ as notation for any fundamental weight corresponding to $\alpha_i$. Thus,  $\eta = -2\nu_{a}$.
\end{remark}
%
%
%
%
%
%
%
%
\subsection{Order of vanishing of the Hasse invariant}\label{unitary.vanishing}
For simplicity we omit the subscript and let $G=\GL(n)_k$ in this section. Let the signature be $(a,a')=(n-1,1)$. It is possible to use \S\ref{general.strategy} to obtain the vanishing order, but we present here an approach via the Pl\"ucker embedding which could be of independent interest. 
Let us quickly recall this embedding and introduce necessary notation.
\subsubsection{The Pl\"ucker embedding}
Let $F$ denote the flag
\begin{equation*}
    F : 0 \subset (e_n) \subset (e_{n-1},e_n) \subset \cdots \subset (e_2, ..., e_n) \subset k^n,
\end{equation*}
where $(e_1, ..., e_n)$ is the standard basis for $k^n$. Let $\text{Flag}_n$ denote the set of full flags in $k^n$. Then $\text{Stab}_G(F) = B$. Since $G$ acts transitively 
on $\text{Flag}_n$, the map $gB\mapsto gF$ induces an isomorphism $G/B\cong \text{Flag}_n$. 

Given an integer $m$ with $1\leq m \leq n-1$, we use the lexicographic order on the basis $\{ e_{i_1}\wedge \cdots \wedge e_{i_m}\}_{i_1<\cdots <i_m}$ of $\wedge^m k^n$. 

If $V\subset k^n$ is an $m$-dimensional subspace, let $[\wedge^m V]\in \PP(\wedge^m k^n)\cong \PP^{\binom{n}{m}-1}$ denote the point whose $i^{\text{th}}$ component is the coefficient of the $i^{\text{th}}$ element in the basis of $\wedge^m k^n$ as described above (well-defined up to a scalar). The Pl\"ucker embedding is the morphism
\begin{equation*}
\begin{aligned}
    \iota : G/B\cong \text{Flag}_n &\hookrightarrow \prod_{i=1}^{n-1} \PP(\wedge^{i}k^n) \cong \prod_{i=1}^{n-1}\PP^{\binom{n}{i}-1} \\
    \Big( 0\subsetneq V_1 \subsetneq \cdots \subsetneq V_{n-1}\subsetneq k^n \Big) &\mapsto \Big( [V_1], [\wedge^ 2V_2], ..., [\wedge^{n-1}V_{n-1}]\Big).
\end{aligned}
\end{equation*}
Let $[x_{0,i} : x_{1,i} : ... : x_{\binom{n}{i}, i} ]$ denote the coordinates of $\PP(\wedge^{i}k^n)$. 
For each $\sigma\in W=S_n$, $\sigma F$ is the flag
\begin{equation*}
    \sigma F : 0\subset (e_{\sigma(n)}) \subset ... \subset (e_{\sigma(n)}, e_{\sigma(n-1)}, ..., e_{\sigma(2)})\subset k^n.
\end{equation*}
For any $b\in B$, $b\sigma F \in B \sigma B/B \subset \prod_{i=1}^{n-1}\PP^{\binom{n}{i}-1}$ is the point whose whose $x_{j,i}$-coordinate is the coordinate of the $j^{\text{th}}$ basis vector when we write out 
\[
\wedge^{i}\Span\{be_{\sigma(n)}, be_{\sigma(n-1)}, ..., be_{\sigma(n-i+1)}\}
\]
as a sum of the basis of $\wedge^{i}k^n$ as described above. 

Let $U$ denote the unipotent radical of $B$. Analogously to the Pl\"ucker embedding, there is an open embedding $\iota_U \colon G/U \hookrightarrow \prod_{i=1}^{n-1} \wedge^{i} k^n \setminus \{ 0 \}$.
\begin{lemma}\label{restricting.coordinate.function}
Suppose that $f$ is a coordinate function on $\prod_{i=1}^{n-1} \PP(\wedge^i k^n)$. Then, for any point $x\in G(k)/B(k)$, we have that $\ord_x(\iota^* f) = \ord_x(f)$.
\end{lemma}
\begin{proof}
Let $\pi_G \colon G/U\to G/B$ and $\pi_{\PP}\colon \prod_i^{n-1} \wedge^{i}k^n \to \prod_{i}^{n-1}\PP(\wedge^{i}k^n)$ denote the projections. The following diagram commutes:
\begin{equation*}
    \begin{tikzcd}
        G/U \arrow[r, "\iota_U"] \arrow[d, "\pi_G"] & \prod_i \wedge^{i}k^n \setminus \{ 0 \} \arrow[d, "\pi_{\PP}"] \\
        G/B \arrow[r, "\iota"] & \prod_i \PP(\wedge^{i}k^n)
    \end{tikzcd}
\end{equation*}
Since $f$ is a coordinate function, $\ord(f)=\ord(\pi_{\PP}^{*}f)$. Since $\iota_U$ is an open immersion and the diagram commutes,
\begin{equation}
    \ord(\iota_U^* \pi_{\PP}^*f)=\ord(\pi_G^*\iota^*f) \geq \ord(\iota^*f) \geq \ord(f) = \ord(\pi_{\PP}^*f) = \ord(\iota_U^* \pi_{\PP}^*f).
\end{equation}
Hence, the inequalities are equalities.
\end{proof}
%
%
%
%
%
%
%
%
\subsubsection{Order of vanishing of highest weight sections}
For each $w\in W$, let $C(w)\coloneqq BwB/B \subset G/B$ denote the stratum corresponding to $w$ under the Bruhat stratification of $G/B$. Let $X(w)\coloneqq \overline{C(w)}$.
\begin{lemma}\label{lemma: hasse vanishing unitary}
Let $f$ be a nontrivial global section of $H^0(w_0\eta)_\eta$. Then the zero locus of $f$ is exactly $X(w_0s_1)$, and for any $k$-point $x$ in $X(w_0s_1)$ we have that $\ord_x(f)=2$.
\end{lemma}
\begin{proof}
%
%
%
%
Since $\eta=-2\nu_{n-1}$, 
\begin{equation*}
    \mathcal{L}_{G/B}(w_0\eta) = \mathcal{L}_{G/B}(2\nu_1) = \calo_{G/B}(2X(w_0s_1)).
\end{equation*} 
Let $Y:=G/B \cap \{ x_{0,n-1}=0 \}$. We claim that 
\begin{equation}\label{x.w.0.y}
    X(w_0s_1) = Y.
\end{equation}
For all $b\in B$ and $\sigma\in W$, a point $b\sigma F$ in $B \sigma B/B$ corresponds to a tuple $$\left( [x_{0,1}:...:x_{n,1}], ..., [x_{0,i}:...:x_{\binom{n}{i},i}],..., [x_{0,n-1} : ... : x_{n,n-1}]\right),$$ where the coordinate $x_{0,n-1}$ is the coefficient of $e_1\wedge e_2 \wedge\cdots\wedge e_{n-1}$ when we write out $be_{\sigma(n)}\wedge be_{\sigma(n-1)}\wedge \cdots be_{\sigma(2)}$ as a sum of the basis vectors $e_{j_1}\wedge\cdots \wedge e_{j_{n-1}}$ (by definition of the image of the largest piece of the filtration $b\sigma F$ under the Pl\"ucker embedding). Notice now that the vector $e_1\wedge e_2 \wedge\cdots\wedge e_{n-1}$ will appear with a non-zero coefficient if and only if $\{\sigma(2),...,\sigma(n)\} = \{1,...,n-1\}$, i.e. if and only if $\sigma(1)=n$. In other words, $b \sigma F \in Y$ if and only if $\sigma(1)\neq n$. 

As a permutation, $w_0(j)=n-j+1$ for all $j\in\{1,2,...,n\}$ and $s_1=(12)$, so $w_0s_1(1)=w_0(2)=n-1\neq n$, whence $C(w_0s_1)\subset Y$. Since $Y$ is closed this implies that $X(w_0s_1)\subset Y$. Since $Y$ is an intersection of equidimensional varieties it is itself equidimensional. Suppose $Y\neq X(w_0s_1)$, then it has an irreducible decomposition $Y=X(w_0s_1)\cup Y_1\cup ... \cup Y_m$ where $\codim Y_i = \codim X(w_0s_1) = 1$ inside $G/B$. Since the codimension one Schubert varieties generate all divisors, $[Y_i]=\sum_{j\in J} a_j [X(w_0 s_j)]$ for some index set $J_i$ (with $\{1\}\neq J_i$ since $Y\neq X(w_0s_1)$) and  integers $a_j$. Thus, $Y_i=\cup_{j\in J_i} X(w_0s_j)$, as this is the support of $[Y_i]$. Bu $w_0s_j(1)=n$ for all $j\neq 1$. Whence $C(w_0s_j)\cap Y=\emptyset$. Thus $X(w_0s_j)\not\subset Y$. This contradicts that $Y_i = \cup_j X(w_0s_j)$. We conclude that $Y=X(w_0s_1)$.

This implies that
\begin{equation*}
\mathcal{L}_{G/B}(\nu_1) = \calo_{G/B}(X(w_0 s_1)) = \calo_{G/B}(Y) = \calo(0,...,0,1).
\end{equation*}
Thus
\begin{equation*}
\mathcal{L}_{G/B}(w_0\eta) = \mathcal{L}_{G/B}(2\nu_1) = \calo(0,...,0,2).
\end{equation*}
So elements of $H^{0}(w_0\eta)$ can be seen as homogeneous polynomials in the $x_{i,n-1}$-coordinates of degree 2. 
Let $h(x)=x_{0,n-1}^2$. 
Let $b=(b_{ij})_{1\leq i,j\leq n}\in B$. By definition of the action of $b$ on $G/B\subset \prod_{i=1}^{n-1} \PP(\wedge^{i}k^n)$ described above, 
\begin{equation*}
\begin{aligned}
(b\cdot h)(x) = h(b^{-1}x) &= h(b_{11}^{-1}b_{22}^{-1}\cdots b_{n-2,n-2}^{-1} x_{0,n-1} \ast, ..., \ast) \\
&= b_{11}^{-2}\cdots b_{n-2,n-2}^{-2}x_{0,n-1}^2 \\
&= \eta'(b)h(x),
\end{aligned}
\end{equation*}
where $\eta' = (-2,-2,...,-2,0)=\eta + (-1,...,-1)$. 
Thus
\begin{equation*}
\langle f \rangle = H^{0}(w_0 \eta)_\eta = \langle x_{0,n-1}^{2} \rangle.
\end{equation*}
By \ref{restricting.coordinate.function}, the vanishing order of $f$ is either $0$ or $2$. By (\ref{x.w.0.y}) the vanishing locus of $f$ is exactly $X(w_0s_{1})$.
\end{proof}
\subsubsection{The Bruhat stratification}
Since $W_I=W_J=\langle s_1, ..., s_{n-2}\rangle$ we see that ${}^{I}W^{J}=\{ 1 , s_{n-1}\}$, where the nondense Bruhat stratum $\pizipb^{-1}(x_{s_{n-1}})\subset \gzip^{\mu}$ is the closure of the Ekedahl-Oort stratum $[E\backslash G_{w_0s_1}]$. The Bruhat stratification of 
$\gzip^{\mu}$ is thus given by 
\[
\gzip^\mu = \pizipb^{-1}(x_1) \bigsqcup [E\backslash \overline{G_{w_0s_1}}]).
\]
The vanishing order of the Hasse invariant $\ha(G,\mu,\wedge^n(\std \oplus \std^{\vee}))$ is described explicitly by:
\begin{theorem}\label{hasse.vanishing.unitary}
Let $(G,\mu)$ be as in \S\ref{section: group notation unitary case}, with $(a,a')=(n-1,1)$. For any $x\in \gzip^{\mu}(k)$, 
if $\ha(G,\mu,\wedge^n(\std \oplus \std^{\vee}))$ 
vanishes at 
$x$, then it vanishes to order 2. The zero locus is exactly 
$\pizipb^{-1}(x_{w_0s_1})$.
\end{theorem}
\begin{proof}
Since $\eta$ is a multiple of the fundamental weight corresponding to $\alpha_{n-1}^{\vee}$ and since $w_{0,I}$ is a product of simple reflections $s_{1}$, ..., $s_{n-2}$, we see that $w_{0,I}\eta = \eta$. Since $-w_0I=J$, conjugation by $w_0$ gives $w_0w_{0,J}w_0=w_{0,I}$. Hence $D_{w_0}(-\eta)=(p-1)\eta$. 
Thus, the order of vanishing of the Hasse invariant can be obtained from that of any generator of 
$H^{0}(\text{Sbt} , \mathcal{L}_{\sbt}(-\eta, w_0\eta)) = H^{0}(w_0\eta)_{\eta}$. 
The result follows from \ref{lemma: hasse vanishing unitary}.
%
%
%
%
%
\end{proof}
%
%
%
%
%
%
%
%
%
\subsection{The conjugate line position}
\begin{proposition}\label{proposition: clp in unitary case arbitrary signature}
Let $(G,\mu)$ be as in \S\ref{section: group notation unitary case}. For each $k$-point $x$ in $\gzip^{\mu}$, we have that \[
\clp_x(G,\mu,\wedge^n(\std\oplus \std^{\vee}))\leq 2\min\{a,a'\}.
\]
\end{proposition}
\begin{proof}
Let $\underline{\Vscr}$ be the $F$-zip over $\gzip^{\mu}$ associated to $\std\oplus \std^{\vee}$. Fix a $k$-point $x$ in $\gzip^\mu$. The descending filtration on the $F$-zip $\wedge^n\Vscr_x$ is induced from the zip-datum of $\underline{\Vscr}$. The $m^{\text{th}}$ graded piece of the descending filtration is given by
\begin{equation*}
    \fil_{\Hdg}^m (\wedge^n\Vscr_x) = \im\Big( (\fil_{\Hdg}^{1}\Vscr_x)^{\otimes m} \otimes (\Vscr_x)^{\otimes n-m} \to (\Vscr_x)^{\otimes n} \to \bigwedge^{n}\Vscr_x \Big),
\end{equation*}
and thus the $m^{\text{th}}$ graded piece of the Hodge filtration is given by 
\begin{equation*}\label{graded.piece}
    \gr^{m}\fil_{\Hdg}^{\bullet}(\wedge^n\Vscr_x) = \bigwedge^{m} \fil_{\Hdg}^1(\Vscr_x) \otimes \bigwedge^{n-k} (\Vscr_x/\fil_{\Hdg}^1(\Vscr_x)).
\end{equation*}
Hence, in a basis $e_1, ..., e_{2n}$ of $\Vscr_x$ such that $e_1, ..., e_n$ is a basis of $\fil_{\Hdg}^1(\Vscr_x)$, $\clp_x(G,\mu,\wedge^n(\std\oplus \std^{\vee}))$ 
is given by the number of the basis vectors $e_{n+1}, ..., e_{2n}$ which are mapped into $\fil_{\Hdg}^1(\Vscr_x)$ by Frobenius, denoted $\fr$.

We have that $\underline{\Vscr}_x = \underline{\Vscr}(G,\mu,\std)_x\oplus \underline{\Vscr}(G,\mu,\std^{\vee})_x$. By definition of $\mu$, we have that $a = \dim \fil_{\Hdg}^1(\Vscr_x)\cap \Vscr(G,\mu,\std)_x$ and $a'=\dim \fil_{\Hdg}^1(\Vscr_x)\cap \Vscr(G,\mu,\std^{\vee})_x$. Let $e_1, ..., e_{2n}$ be a basis of $\Vscr_x$ such that $e_1, ..., e_a$ is a basis of $\fil_{\Hdg}^1(\Vscr_x)\cap \Vscr(G,\mu,\std)_x$ and $e_{a+1}, ..., e_n$ is a basis of $\fil_{\Hdg}^1(\Vscr_x)\cap \Vscr(G,\mu,\std^{\vee})_x$. By \cite{moonen1} we may assume that for any $i=1,...,2n$, $\fr(e_i)= e_j$ for some $j \leq i$ or $\fr(e_i)=0$. 
Since $a=\dim \fil_{\Hdg}^1(\Vscr_x)\cap \Vscr(G,\mu,\std)_x$, we have that
\begin{equation}\label{bound1.a'}
    \#\{ e_i \in \Vscr(G,\mu,\std)_x : e_i\notin \fil_{\Hdg}^1(\Vscr_x) \text{ and } \fr(e_i) \in \fil_{\Hdg}^1 \cap \, \Vscr(G,\mu,\std)_x \} \leq a'.
\end{equation}
Since $\ker(\fr)=\fil_{\Hdg}^1(\Vscr_x)$ we also have that
\begin{equation}\label{bound2.a'}
    \#\{ e_i \in \Vscr(G,\mu,\std^{\vee})_x : e_i\notin \fil_{\Hdg}^1(\Vscr_x) \text{ and }\fr(e_i) \in \fil^1 \cap \Vscr(G,\mu,\std^{\vee})_x \} \leq a'.
\end{equation}
Indeed, if $e_{i_1}, ..., e_{i_{a'+1}} \in \Vscr(G,\mu,\std^{\vee})_x\setminus \fil_{\Hdg}^1(\Vscr_x)$ are mapped to $\fil_{\Hdg}^1(\Vscr_x)$, then $\fr(e_{i_1}), ..., \fr(e_{i_{a'+1}})$ are linearly dependent, and thus there exists $x_j \in k$ such that $\sum_j x_j \fr(e_{i_j})=0$, i.e. such that $\sum_{j=1}^{a'+1} x_{j}^{1/p} e_{i_j} \in \ker(\fr) = \fil_{\Hdg}^1(\Vscr_x)$, which is a contradiction. By a similar argument we can replace $a'$ by $a$ in (\ref{bound1.a'}) and (\ref{bound2.a'}) to see that 
\begin{equation*}
    \clp_x(G,\mu,\wedge^n(\std \oplus \std^{\vee}))\leq 2\min\{a,a'\}.
\end{equation*} 
\end{proof}
%
%
\begin{corollary}\label{a.number.unitary}
Let $(G,\mu)$ be as in \S\ref{section: group notation unitary case} and suppose that the signature is $(n-1,1)$. For each $k$-point $x$ in $\gzip^{\mu}$, we have that
\begin{equation*}
    \clp_x(G,\mu,\wedge^n(\std \oplus \std^{\vee}))=0, \,\,\, \text{or} \,\,\, \clp_x(G,\mu,\wedge^n(\std \oplus \std^{\vee}))=2.
\end{equation*}
\end{corollary}
\begin{proof}
Since $\underline{\Vscr}(G,\mu,\std\oplus \std^{\vee})_x = \underline{\Vscr}(G,\mu,\std)_x\oplus \underline{\Vscr}(G,\mu,\std^{\vee})_x$ and since $\underline{\Vscr}(G,\mu,\std^{\vee})_x$ is the dual of $\underline{\Vscr}(G,\mu,\std)_x$, this follows from the previous proposition.
\end{proof}
\begin{theorem}[\ref{th-main}\ref{item-GL}]\label{ogus.principle.unitary.case}
Let $(G,\mu)$ be as in \S\ref{section: group notation unitary case}, with $(a,a')=(n-1,1)$. Then Ogus' principle holds for $(G,\mu,\wedge^n(\std\oplus \std^{\vee}))$.
\end{theorem}
\begin{proof}
This follows by combining \ref{hasse.vanishing.unitary} with \ref{a.number.unitary}.
\end{proof}
\subsection{Application to unitary Shimura varieties}
Let $\gx$ be the Shimura datum with group $\mathbf{G}$. Assume that $(a,a')=(n-1.1)$. Applying \ref{sec-intro-Hodge-type} to $\gx$ gives the special $k$-fiber $S_K$ of the associated Shimura variety and a smooth surjective Zip period map $\zeta_K\to \gzip^{\mu}$. Let $\ha_{S_K}\coloneqq \zeta_K^*\ha(G,\mu,\wedge^n(\std\oplus \std^{\vee}))$.
As in \S\ref{sec-orthogonal-case}, we deduce immediately the following statement for unitary Shimura varieties.
\begin{corollary}
For any $k$-point $s$ in $S$ we have that $\ord_s(\ha_{S_K})\in \{0,2\}$ and $\ord_s(\ha_{S_K})\neq 0$ if and only if $s\in S_{w_0s_1}^{\bruh}$. Furthermore, Ogus' principle holds; for every $x \in S_K(k)$,
$$\ord_x \zeta^*\ha(G,\mu,r)=\clp_x(G,\mu,r).$$ 
\end{corollary}
\begin{remark}
If $\mathbf{G}_\mr\cong \GU(n-1,1)$, then $\ha_{S_K}$ is the classical Hasse invariant associated to de Rham cohomology of the universal abelian scheme over $S_K$. 
\end{remark}
%
%
%
%
%
%
\section{Ogus Principle in type \texorpdfstring{$\Csf_n$}{C}}
\label{siegel.case}
\subsection{Explicit description of the Hasse invariant}
Let $(V,\psi)$ be a symplectic space of dimension $2n$ over $k$.  Let $G\coloneqq \GSp(V,\psi)$ denote the corresponding symplectic group. 
Let $J$ denote the $n\times n$-matrix $J = \text{antidiag}(1,...,1)$. Choose a basis such that $\psi$ is the symplectic form determined by the block matrix $\text{antidiag}(-J,J)$. Let $v_1,...,v_{2n}$ be the standard basis vectors for $k^{2n}$. %
Let $\mu$ be the cocharacter $z\mapsto \diag(zI_n, z^{-1}I_n)$. Under the identification $V =k^{2n}$ the two parabolic subgroups determined by $\mu$ are given by $P = \stab_G\langle v_{n+1},..., v_{2n} \rangle$ and $Q =\stab_G\langle v_1, ..., v_n\rangle$.

The Hasse invariant of $(G,\mu,\wedge^n \std^{\vee})$ is explicitly given (cf. \cite[Section 8.1.2]{koskivirta.imai}) as a function on $G$ by
\begin{equation*}
    \ha(G,\mu,\wedge^n \std^{\vee}) \colon \begin{bmatrix}
        A & B \\ C & D
    \end{bmatrix} \mapsto \det(A).
\end{equation*}
For simplicity, let $\ha(G,\mu)\coloneqq \ha(G,\mu,\wedge^n \std^{\vee})$. 
\subsection{Order of vanishing of the Hasse invariant}
Viewing all $2n\times 2n$-matrices as the affine space $\AA^{(2n)^2}$,  extend $\ha(G,\mu)$ to a function $\widetilde{\ha}(G,\mu)$ on $\AA^{(2n)^2}$. Since $\ord_x(\widetilde{\ha}(G,\mu)) = n-\rk(A)$ for every $k$-point $x=\begin{bmatrix}
        A & B \\ C & D
    \end{bmatrix}\in \AA^{(2n)^2}$, we proceed to show that $\ord_x(\ha(G,\mu)) = \ord_x(\widetilde{\ha}(G,\mu))$ for all $x\in G(k)$.
\begin{notation}
For any $2n\times 2n$-matrix $X$, let $X_{A}$ denote the upper left $n\times n$-matrix of $X$. That is, $X = \begin{bmatrix}
    X_A & \ast \\ \ast & \ast
\end{bmatrix}$.
\end{notation}
\begin{lemma}\label{bruhat.stratum.ranks}
Two elements $X,Y \in \GSp(2n,k)$ lie in the same Bruhat stratum if and only if $\rk(X_A)=\rk(Y_A)$.
\end{lemma}
\begin{proof}
If $X$ and $Y$ lie in the same Bruhat stratum, then there are $p\in P(k)$ and $q\in Q(k)$ such that $pXq=Y$. Since $p_A,q_A$ are invertible we see that $\rk Y_A = \rk(pXq)_A = \rk X_A$. To prove the converse, it suffices to find representatives $X_0,...,X_n$ of different Bruhat strata such that $\rk(X_{0,A})< ...< \rk (X_{n,A})$. For example, if $A_i$ denotes the $n\times n$-diagonal matrix $\diag(1,...,1,0,...,0)$ whose first $i$ entries along the diagonal are $1$ and the rest are 0, then with
\begin{equation*}
    X_i \coloneqq \begin{bmatrix}
         A_i & -J \\ J & 0 
    \end{bmatrix}
\end{equation*}
we find such $X_0,...,X_n$.
\end{proof}
\begin{lemma}\label{hasse.inv.bruhat.stratum}
For any $x_1,x_2$ lying in the same Bruhat stratum, we have that
\begin{equation*}
    \ord_{x_1}(\ha(G,\mu))=\ord_{x_2}(\ha(G,\mu)).
\end{equation*}
\end{lemma}
\begin{proof}
Let $(p,q)\in P\times Q$ such that $px_1q = x_2$. Left and right multiplication $g\mapsto pgq$ gives an isomorphism of schemes $\psi_{p,q}\colon \GSp(2n)\to \GSp(2n)$. For any $x\in \GSp(2n)(k)$, we thus have that 
\begin{equation}\label{ord.psi1}
\ord_x(\psi_{p,q}^*\ha(G,\mu))=\ord_{\psi_{p,q}(x)}(\ha(G,\mu)).    
\end{equation}
We claim that 
\begin{equation}\label{ord.psi2}
\ord_x(\psi_{p,q}^*\ha(G,\mu))=\ord_x(\ha(G,\mu)).
\end{equation}
Indeed, if $k[\GSp(2n)]$ denotes the coordinate ring of $\GSp(2n)$, then $\ha(G,\mu)$ is the image of the element $\det(x_{ij})_{1\leq i,j\leq n}$ under the natural morphism $k[x_{ij}]_{1\leq i,j \leq 2n}\to k[\GSp(2n)]$  and $\psi_{p,q}^*\ha(G,\mu)$ is the image of the element $\det(p_Aq_A)\det(x_{ij})_{1\leq i,j\leq n}$. As they only differ by multiplication with $\det(p_Aq_A)\in k^{\times}$ we see that their order of vanishing is the same at all points. In particular, with $x=x_1$ we get from combining (\ref{ord.psi1}) and (\ref{ord.psi2}) that
\begin{equation}
    \ord_{x_2}(\ha(G,\mu)) = \ord_{\psi_{p,q}(x_1)}(\ha(G,\mu)) = \ord_{x_1}(\psi_{p,q}^*\ha(G,\mu)) = \ord_{x_1}(\ha(G,\mu)).
\end{equation}
%
\end{proof}
\begin{theorem}
\label{vanishing.siegel.case}
The vanishing order of $\ha(G,\mu)=\ha(G,\mu,\wedge^n \std^{\vee})$ precisely assumes the values $0,1,...,n$. Two points admit the same vanishing order if and only if they lie in the same Bruhat stratum.
\end{theorem}
\begin{proof}
By \ref{bruhat.stratum.ranks} and \ref{hasse.inv.bruhat.stratum} it suffices to find points $x_0, ..., x_n \in \GSp(2n,k)$ lying in pairwise distinct Bruhat strata such that $\ord_{x_i}(\ha(G,\mu))=n-i$. Let $\psi\colon \AA^n \to \GSp(2n)$ denote the composition of the morphisms
\begin{equation*}
\begin{aligned}
    \AA^n &\hookrightarrow \prod_{i=1}^{n} \SL(2) \\
    (a_1, ..., a_n)&\mapsto \Big(\begin{bmatrix} a_1 & -1 \\ 1 & 0 \end{bmatrix}, ...,
    \begin{bmatrix}
        a_n & -1 \\
        1 & 0
    \end{bmatrix}\Big)
\end{aligned}
\end{equation*}
and
\begin{equation*}
    \begin{aligned}
        \prod_{i=1}^{n} \SL(2) &\hookrightarrow \GSp(2n) \\
        \Big(\begin{bmatrix}
            a_1 & b_1 \\ c_1 & d_1
        \end{bmatrix}, ...
        \begin{bmatrix}
            a_n & b_n \\ c_n & d_n
        \end{bmatrix}\Big) &\mapsto \begin{bmatrix}
            \diag(a_1,...,a_n) & \diag(b_1,...,b_n)J \\
            \diag(c_n,...,c_1)J & \diag(d_n, ..., d_1)
        \end{bmatrix}.
    \end{aligned}
\end{equation*}
Let $\iota\colon \GSp(2n)\hookrightarrow \AA^{(2n)^2}$ denote the inclusion. 
For every $a\in \AA^n(k)$, 
\begin{equation}\label{inequalities.squeezing.ha}
    \ord_a(\psi^*\ha(G,\mu)) \geq \ord_{\psi(a)}(\ha(G,\mu)) \geq \ord_{\iota(\psi(a))}(\widetilde{\ha}(G,\mu)).
\end{equation}
Since $\psi^*\ha(G,\mu)$ is given by $a \mapsto \prod_i a_i$,
\begin{equation*}
    \ord_a(\psi^*\ha(G,\mu)) = \Big|\{ i : a_i = 0 \}\Big| = n-\rk(\psi(a)_A) = \ord_{\iota(\psi(a))}(\widetilde{\ha}(G,\mu)).
\end{equation*}
So equality holds in 
\eqref{inequalities.squeezing.ha}. 
By \ref{bruhat.stratum.ranks}, the images $x_0=\psi(0)$ and $x_i\coloneqq \varphi(e_1+ \ldots +e_i)$ have the desired property.
\end{proof}
\begin{theorem}[\ref{th-main}\ref{item-GSp}]\label{corollary: main theorem siegel case}
The triple $(G,\mu,\wedge^n\std^{\vee})$ satisfies Ogus' principle.
\end{theorem}
\begin{proof}
This follows from \ref{vanishing.siegel.case} and \cite[Example 2.4]{wedhorn.bruhat2}.
\end{proof}
\subsection{Application to Siegel modular varieties}
\begin{corollary} 
The Siegel modular variety $S_{g,K}$~\eqref{sec-intro-Hodge-type} satisfies Ogus' Principle. 
The order of vanishing of $\ha_{S_{g,K}}$  exactly assumes the values $0,1,...,n$. Two points have the same order of vanishing if and only if they lie in the same Bruhat stratum.
\end{corollary}
\begin{proof}
This follows from \ref{corollary: main theorem siegel case} and \ref{vanishing.siegel.case}.
\end{proof}
\subsection{Beyond Shimura varieties: The \texorpdfstring{$(\Csf_n, \Csf_{n-1})$}{(Cn,Cn-1)}-case}
Let $G=\Sp(2n)$ and let $\mu=e_1$ be the fundamental weight corresponding to the simple root $e_1-e_2$. The pair $(W, W_I)$ is the same as in the orthogonal $\Bsf_n$-case. Neither the pair $(G,\mu)$, nor the adjoint pair $(G^{\ad}, \mu^{\ad})$ arises from a Shimura variety, because $\mu^{\ad}$ is cominuscule but not minuscule.
\begin{theorem}[\ref{th-intro-Cn-Cn-1}]
\label{th-Cn-Cn-1}
Let $G=\Sp(2n)$, $\mu=e_1$ and
 and $r=\std$. Then Ogus' holds for $(G,\mu,r)$ at all strata but the minimal one. More precisely, let $w \in \iw$ and $ x \in \GZip^{\mu, \EO}_w(k)$. Then:
 \begin{enumerate}
\item If  $w \neq e,z$, then Ogus' Principle holds: $\ord_x \ha(G,\mu,r) = \clp_x(G,\mu,r)=1$;
\item  If $w=z$, then Ogus' Principle holds:     $\ord_x \ha(G,\mu,r) = \clp_x(G,\mu,r)=0$;
\item If $w=e$, then $\ord_x \ha(G,\mu,r)=1$ but $\clp_x(G,\mu,r)=2$ so Ogus' Principle fails but the inequality~\eqref{eq-intro-ineq} holds.
 \end{enumerate}
 In particular, the divisor of the Hasse invariant $\ha(G,\mu,r)$ is reduced.
 \end{theorem}
\begin{proof}
    This follows in the same way as the $(\Bsf_n, \Bsf_{n-1})$ case~\ref{corollary: ogus principle for Bn}, except that the root-pairings with value $2$ in that case are $1$ here.
\end{proof}
\begin{remark}
It is interesting that the Hasse invariant exhibits more regular behavior in the non-Shimura case $(\Csf_n,\Csf_{n-1})$ than in the case $(\Bsf_n, \Bsf_{n-1})$ which does arise from orthogonal Shimura varieties.
   
\end{remark}
%
%
%
%
%
%
%
%
%
%
%
%
%
%
%
%
\printbibliography 
\end{document}